\documentclass[journal,twoside,web]{ieeecolor}
\usepackage{generic}
\usepackage{cite}
\usepackage{amsmath,amssymb,amsfonts}
\usepackage{algorithmic}
\usepackage{graphicx}
\usepackage{textcomp}
\usepackage{cite}
\usepackage{amsmath,amssymb,amsfonts}
\usepackage{algorithm}
\usepackage{algorithmic}
\usepackage{graphicx}
\usepackage{textcomp}
\usepackage{bm}
\usepackage{mathrsfs}% Enable this line and disable the
\usepackage{graphicx,cite}
\usepackage{amsmath}
\usepackage{amssymb}
\usepackage[dvips]{epsfig}
\usepackage{latexsym}
\usepackage{psfrag}
\usepackage{graphicx,subfigure}
\usepackage{hyperref}
\usepackage{color}
\usepackage{amsfonts,mathrsfs,bbm,amsmath,stmaryrd,amssymb,pifont}%
\usepackage{amscd,graphicx,array,dsfont,texdraw,tikz}%
\usetikzlibrary{arrows,decorations.markings,shapes,shadows,fadings}
\usepackage{bbm,mathbbol,mathrsfs}%
\usepackage{array}
\usepackage{enumerate}
\usepackage{times}
\usepackage{ntheorem}

\newtheorem{definition}{Definition}%[section]
\newtheorem{lemma}{Lemma}%[section]
\newtheorem{theorem}{Theorem}%[section]
\newtheorem{remark}{Remark}%[section]
\newtheorem{assumption}{Assumption}%[section]
\def\dref#1{(\ref{#1})}

\allowdisplaybreaks

\def\BibTeX{{\rm B\kern-.05em{\sc i\kern-.025em b}\kern-.08em
    T\kern-.1667em\lower.7ex\hbox{E}\kern-.125emX}}
\begin{document}
%\title{Distributed Resource Allocation over Interacting MANs: A Game-Theoretic Approach}
%\title{Distributed Resource Allocation over Multiple Interacting Coalitions}
\title{DRAG: Distributed Resource Allocation Games over Multiple Interacting Coalitions}
\author{Jialing Zhou, \IEEEmembership{Member, IEEE}, Guanghui Wen, \IEEEmembership{Senior Member, IEEE},  Yuezu Lv, \IEEEmembership{Member, IEEE}, Tao Yang, \IEEEmembership{Senior Member, IEEE}, Guanrong Chen, \IEEEmembership{Life Fellow, IEEE}
\thanks{J. Zhou and Y. Lv are with the Advanced Research Institute of Multidisciplinary Science, Beijing Institute of Technology, Beijing 100081, China (e{-}mail: jlzhou@bit.edu.cn, yzlv@bit.edu.cn).}
\thanks{G. Wen is with the School of Automation, Southeast University, Nanjing 211189, China (e{-}mail: wenguanghui@gmail.com).}
\thanks{T. Yang is with the State Key Laboratory of Synthetical Automation for Process Industries, Northeastern University, Shenyang 110819, China (e{-}mail: yangtao@mail.neu.edu.cn).}
	\thanks{G. Chen is with the Department of Electrical Engineering, City University of Hong Kong, Hong Kong SAR, China (e{-}mail: eegchen@cityu.edu.hk).}
%\thanks{Jinhu L\"u is with the School of Automation Science and Electrical Engineering, Beihang University, Beijing 100191, China (e{-}mail: jhlu@iss.ac.cn).}
%	\thanks{Xinghuo Yu is with		the School of Engineering, RMIT University, Melbourne, VIC 3000, Australia (e{-}mail: x.yu@rmit.edu.au).}
%	\thanks{Corresponding author: Guanghui Wen}
}
%\author{First A. Author, \IEEEmembership{Fellow, IEEE}, Second B. Author, and Third C. Author, Jr., \IEEEmembership{Member, IEEE}
%\thanks{This paragraph of the first footnote will contain the date on
%which you submitted your paper for review. It will also contain support
%information, including sponsor and financial support acknowledgment. For
%example, ``This work was supported in part by the U.S. Department of
%Commerce under Grant BS123456.'' }
%\thanks{The next few paragraphs should contain
%the authors' current affiliations, including current address and e-mail. For
%example, F. A. Author is with the National Institute of Standards and
%Technology, Boulder, CO 80305 USA (e-mail: author@boulder.nist.gov). }
%\thanks{S. B. Author, Jr., was with Rice University, Houston, TX 77005 USA. He is
%now with the Department of Physics, Colorado State University, Fort Collins,
%CO 80523 USA (e-mail: author@lamar.colostate.edu).}
%\thanks{T. C. Author is with
%the Electrical Engineering Department, University of Colorado, Boulder, CO
%80309 USA, on leave from the National Research Institute for Metals,
%Tsukuba, Japan (e-mail: author@nrim.go.jp).}}

\maketitle

\begin{abstract}
	 Despite many distributed resource allocation (DRA) algorithms have been reported in literature, it is still unknown how to allocate the resource optimally over multiple interacting coalitions.   One major challenge in solving such a problem is that, the relevance of the decision on resource allocation in a coalition to the benefit of others may lead to conflicts of interest among these coalitions. Under this context, a new type of multi-coalition game is formulated in this paper, termed as resource allocation game, where each coalition contains multiple agents that cooperate to maximize the coalition-level benefit while  subject to the resource constraint described by a coupled equality. 	Inspired by techniques such as variable replacement, gradient tracking and leader-following consensus, two new kinds of DRA algorithms are developed respectively  for  the scenarios where the individual benefit of each agent  explicitly depends on the states of   itself and some agents in other coalitions, and on the states of all the game participants. It is shown that the proposed algorithms can converge linearly to the Nash equilibrium (NE) of the multi-coalition game
  while satisfying the resource constraint during the whole NE-seeking process. Finally, the validity of the present allocation algorithms is verified by numerical simulations. 	
\end{abstract}

\begin{IEEEkeywords}
	%Enter key words or phrases in alphabetical order, separated by commas. For a list of suggested keywords, send a blank e{-}mail to keywords@ieee.org or visit \underline {http://www.ieee.org/organizations/pubs/ani\_prod/keywrd98.txt}
	Distributed resource allocation,  distributed NE seeking, distributed  optimization, multi-agent system, multi{-}coalition game.
\end{IEEEkeywords}
%%%%%%%%%%%%%%%%%%%%%%%%%%%%%%%%%%%%%%%%%%%%%%%%%%%%%%%%%%%%%%%%%%%%%%%%%%%%%%%%%%%%%%%%%%

\section{Introduction}
%%%%%%%%%%%%%%%%%%%%%%%%%%%%%%%%%%%%%%%%%%%%%%%%%%%%%%%%%%%%%%%%%%%%%%%%%%%%%%%%%%%%%%%%%%
%Distributed resource allocation has attracted
\IEEEPARstart{T}{he} past decade has witnessed a significant progress on % the development of tremendous distributed resource allocation algorithms
distributed resource allocation (DRA) over multi-agent networks (MANs), where interacting individual agents %subject to network topologies
cooperate to make the best decision
%in a distributed manner
on allocating the group-level resources   via information exchange among neighboring agents \cite{survey-yangtao}. %\cite{push-sum-yangtao,gfh-cluster-dispatch,discrete-time-finite-time1,Initialization-free2016,guanghui_tii,tnse-zhuyanan,ChenLiAuto2018ConDis,zjl-dispath-tnse,chengang-predefined-time-DRA2022,discrete-dispatch 2014 TPWRS,discrete-dispatch 2014 TCNS,discrete-dispatch 2013 TSP,discrete-dispatch 2014 TAC}.
The %problem of distributed resource allocation can be
task can be basically modeled as a distributed optimization problem regarding a group-level objective function
%The basic problem formulation of this task is  a distributed optimization problem regarding a group-level objective function,
%in the form of summation of the individual objective functions,
while
subject to the resource constraint described by a coupled equality, and the problem has been extensively studied %in existing literature
from various aspects with  discrete-time \cite{push-sum-yangtao,gfh-cluster-dispatch,discrete-dispatch-2014-TPWRS,discrete-dispatch-2014-TCNS,discrete-dispatch-2013-TSP,discrete-dispatch-2014-TAC,discrete-time-finite-time1} and continuous-time  \cite{Initialization-free2016,guanghui_tii,tnse-zhuyanan,ChenLiAuto2018ConDis,zjl-dispath-tnse,chengang-predefined-time-DRA2022} DRA algorithms developed. % To solve this problem, tremendous distributed algorithm are developed in both discrete-time and continuous-time setting.
%where unifiy is designed for the group of interacting individual agents under the infomation albility constraint, such that the agent states will converge to the best decision of resource allocation that maximize the group benefit.
%distributed optimization algorithms are developed for  from various aspects

%In some situations, the objective function may explicitly influenced by other coalitions' resource allocation.
%Due to complex interactions in real-world networked systems, the resource allocation problems of different MANs may be coupled with each other.
Considering the complex interactions of real-world networked systems, the resource allocation problems of multiple coalitions may be coupled with each other.
%Ê¡Êо­¼Ã·¢Õ¹%the economic development of a region depends on the allocation
%For example, in  public finance management, when deciding the allocation of a provincial government's revenue fund for economic development, the influence of other provinces' economic development would be taken into account, since cooperation and competition may exist across provinces.
For example, in public finance management, when  deciding the allocation of a provincial government's revenue fund for economic development, the influence of other provinces' economic development would be taken into account, since cooperation and competition may exist across provinces.
% due to the possible existence of cross-province cooperation and competition.
%for cross-province cooperation and competition.
%
%For example, in  public finance management, when deciding the allocation from a provincial government's revenue fund to municipalities  for economic development, the influence of other provinces' economic development would be taken into account, since cooperation and competition may exist across provinces.
%»ú³¡º½Â··ÖÅä
In such cases, %where the objective function may explicitly influenced by other MAN's resource allocation,
the DRA problem of multiple coalitions cannot be decoupled into several independent single-coalition DRA problems and solved separately by employing existing DRA algorithms.

Inspired by the above observations,
a new model is formulated in this paper for the resource allocation problem of multiple interacting coalitions.
In this model, the inputs of each individual agent's objective function may include the states of  agents not only within but also outside the coalition.
The group-level objective function of each coalition is the sum of objective functions of all the individual agents therein. In each coalition, the individual agents cooperate to minimize the group-level objective function while subject to the resource constraint described by a coupled equality.

The proposed model can be viewed as a new type of multi-coalition game, as it shares the core feature of
%existing multi-coalition games, i.e., the characterization of
capturing the cooperation of agents that belongs to the same coalition and the conflicts of interest among different coalitions.
Existing studies on multi-coalition games can be found in \cite{two-network,two-network-switching,multi-cluster-ymj2018,Pang2022Auto,nianxiaohong2021,ZENG2019,mengmin2020,jialing2022arxiv-multi-coalition} and the reference therein, which can be generally classified into two categories according to whether or not intra-coalition consistency constraints are involved.
Specifically, the  agent states in each coalition are free of coupled constraints in one category of studies \cite{multi-cluster-ymj2018,Pang2022Auto,nianxiaohong2021},  while in the other, are demand to reach an agreement \cite{two-network,two-network-switching,ZENG2019,mengmin2020,jialing2022arxiv-multi-coalition}.
%Specifically, the states of individual agents in the same coalition are free of coupled constraints in one category of studies \cite{multi-cluster-ymj2018,Pang2022Auto,nianxiaohong2021}, and while in the other, are demand to reach an agreement \cite{two-network,two-network-switching,ZENG2019,mengmin2020,jialing2022arxiv-multi-coalition}.  %One is without any coupled contriants within coalitions []. The other is to consider consistency constraints of coalition members [].

In the  seminal work on  games of MANs with coalition structure  \cite{two-network},  a two-network zero-sum game is formulated, where the agents within each network should agree on a common network state, and the two networks have opposite objective regarding
the optimization of a common function of the network states;  for this model,
distributed Nash equilibrium (NE) seeking algorithms  are developed under fixed and switching topologies respectively in \cite{two-network} and \cite{two-network-switching}. Then, the work is extended to non-cooperative games among multiple coalitions in \cite{multi-cluster-ymj2018} with the intra-coalition demand of state consistency  removed, and a continuous-time NE computation algorithm is designed based on gradient play and average consensus protocol. Along this line, directed and switching topologies are further considered in \cite{nianxiaohong2021}, and discrete-time gradient-free algorithm design for the case with unknown expressions of objective functions is studied in \cite{Pang2022Auto}.
For multi-coalition games with intra-coalition consistency constraints, a generalized nonsmooth distributed NE seeking algorithm is designed with continuous-time setting in \cite{ZENG2019}; while discrete-time algorithms are developed under undirected and directed network topologies in \cite{mengmin2020} and \cite{jialing2022arxiv-multi-coalition} respectively. It is worth noting that,
although  multi-coalition games have been studied with several NE seeking algorithms developed, research on the proposed multi-coalition game  with coupled equality constraint in the context of resource allocation  has not been reported yet.

%Although some types of multi-coalition games have been formulated and studied in existing literature, the model in this paper have not been reported yet.
%Multi-coalition games have been studied.
%however, the game with equality constraint that describes the resource allocation within each coalition have not been reported.
%the game with equality constraint within the context of resource allocation  have not been reported.

%for the case that the welfare of each coalition (subnetwork) may be influenced by the decision of allocation of other subnetworks.

In this paper, we propose a new model of multi-coalition game for the study of DRA over multiple interacting coalitions.
%For the special yet general case that the individual benefit of each agent is explicitly influenced by the states of itself and some agents in other coalitions,  we design a DRA algorithm based on the techniques of variable replacement, gradient descent, and leader-following consensus. For the general case that the individual benefit of each agent may depends explicitly on the states of all the game participants, the proposed DRA algorithm is redesigned based on  gradient tracking. The proposed algorithms are shown to converge linearly to the NE of the proposed game while meet the equality constraints during the iterations.
For the case that the individual benefit of each agent is explicitly influenced by the states of itself and some agents in other coalitions,   a new kind of DRA algorithm is designed based upon the techniques of variable replacement, gradient descent, and leader-following consensus.  Then, the more general case is further investigated by redesigning the proposed DRA algorithm based on the gradient tracking technique,  where the individual benefit of each agent is allowed to depend explicitly on the states of all game participants.  The proposed algorithms are theoretically proven  to converge linearly to the NE of the proposed game while meet the equality constraints during the iterations.

%
%The main contribution of this paper lies in the following aspects: (i) A new model of multi-coalition game is proposed, which captures the  cooperation of individual agents on resource allocation in each coalition as well as the conflicts of interest among different coalitions. Our model includes the problem of DRA over a single coalition as a special case.
%(ii) Two new kinds of DRA algorithms are designed, which can converge linearly to the NE of the resource allocation game. Another feature of the proposed algorithms is that, the  resource constraints are  satisfied at each iteration, which enables the proposed algorithms to online solve DRA problems such as economic dispatch over multiple interacting coalitions.
%(iii) The proposed algorithms will degenerate into single-coalition DRA algorithms when there is only one coalition in the proposed model. The methodology developed in this paper generalizes the existing results on DRA over MANs.

The main contribution of this paper lies in the following aspects: (i) A new model of multi-coalition game is proposed, which captures the  cooperation of individual agents on resource allocation in each coalition as well as the conflicts of interest among different coalitions. Our model includes the commonly studied mathematical model for DRA  problem over MANs as a special case where all the agents  are assumed to be cooperative. %
(ii) Two new kinds of DRA algorithms are designed and utilized such that the decisions of the agents can converge linearly to the NE of the considered resource allocation game. The methodology developed in this paper generalizes the existing results on DRA over MANs, as the developed algorithms could deal with the DRA problem in the presence of conflicts of interests among different coalitions.
(iii)  Another distinguished feature of the proposed algorithms is that the resource constraints can be guaranteed at each iteration, which enables the proposed algorithms to be executed in an online manner. Such a feature plays an important role in online solving various DRA problems or their variations such as the distributed economic dispatch problem of smart grid with multiple generating units subject to the constraint of supply-demand balance.

%while reach agreements on the network state. The agents in each network  and the agents in each network aim to settle on the common network state that optimize the objective function.

The remainder of the paper is summarized as follows. In Section \ref{sec.pro},  the model of the  game is formulated and  the  property of the NE is analyzed. In Section \ref{sec.special case} and \ref{sec.general case}, DRA algorithms are developed for the special and general cases of the model, respectively.  Numerical examples are provided in Section \ref{sec.simulation} to verify the effectiveness of the proposed algorithms, and finally Section \ref{sec.conclusion} concludes the paper.

\textbf{Notations}.  The sets of  natural numbers, positive integers and real numbers are respectively represented by  $\mathbb{N}$ and $\mathbb{N}^+$ and $\mathbb{R}$. The set of $n${-}dimensional real column vectors is denoted by $\mathbb{R}^n$.  $I_n$ and $\bm{1}_n$ are respectively   the $n${-}dimensional identity matrix and the $n${-}dimensional column vector with all the entries being $1$.  Symbol ${\otimes}$ is the Kronecker product and $\|\cdot\|$ denotes the Euclidian norm.
% $\mathrm{diag}\{B_1,\cdots,B_n\}$ represents the diagonal block matrix with the matrice $B_1,\cdots,B_n$ aligned along the  diagonal blocks.
$\mathrm{diag}\{B_1,\cdots,B_n\}$ represents the diagonal block matrix with the matrix $B_i~(i=1,\cdots,n)$ on the $i$th  diagonal block.
%For a set of real numbers $\{b_1,\cdots,b_n\}$, $\max\{b_1,\cdots,b_n\}$ and $\min\{b_1,\cdots,b_n\}$ respectively denote the maximum and  minimum values among these numbers.

%%%%%%%%%%%%%%%%%%%%%%%%%%%%%%%%%%%%%%%%%%%%%%%%%%%%%%%%%%%%%%%%%%%%%%%%%%%%%%%%%%%%%%%%%%
\section{Problem Statement}\label{sec.pro}
%%%%%%%%%%%%%%%%%%%%%%%%%%%%%%%%%%%%%%%%%%%%%%%%%%%%%%%%%%%%%%%%%%%%%%%%%%%%%%%%%%%%%%%%%%

%Consider a multi{-}coalition game with $N(N\in\mathbb{N}^{+})$ clusters indexed by $1,\cdots, N$, respectively, where the $i$th cluster contains $n_i$ agents.
%\subsection{The underlying communication topology}
%\begin{figure}
%\centering
%\includegraphics[width=3.25in]{multi_layec_topology}
%\caption{The underlying communication topology of the multi{-}coalition game.}\label{fig.graph.sec.2}
%\end{figure}
\subsection{Game Formulation}\label{sec.game formulation}
In this paper, we consider a class of DRA problems  of multiple interacting coalitions indexed by $i\in\mathcal{I}{=}\{1,\cdots,N\}$, where $N\in\mathbb{N}^{+}$ denotes the number of coalitions.  %game of resource allocation among multiple coalitions.
Let $\mathcal{V}_i=\{ij| j=1,\cdots,n_i\}$ be the agent set of coalition $i$, with $ij$ representing the $j$th member in coalition $i$ and  $n_i\in\mathbb{N}^{+}$ denoting the number of the coalition members. Denote the total number of the agents in these coalitions  by $n_{\mathrm{sum}}=\sum_{i=1}^Nn_i$  and  the agent set of the problem by $\mathcal{V}=\mathcal{V}_1\cup\cdots\cup\mathcal{V}_N$. The underlying communication topology among these individual agents is described by an undirected graph $\mathcal{G}(\mathcal{V},\mathcal{E})$.
%The underlying communication topology of the agents in $\mathcal{V}$ is described by an  undirected graph $\mathcal{G}(\mathcal{V},\mathcal{E})$ with the node (agent) set $\mathcal{V}$ and the edge (communication link) set $\mathcal{E}\subseteq \mathcal{V}\times \mathcal{V}$.
%Let $i\in\mathcal{I}~(N\in\mathbb{N}^{+})$ be the index of the coalitions and $n_i$ be the number of agents in coalition $i$. Denote $\mathcal{V}_i=\{i1,i2,\cdots,in_i\}$ as the agent set of coalition $i$, with $ij$ representing the $j$th agent in coalition $i$, and Define the agent set of the game $\mathcal{V}=\mathcal{V}_1\cup\cdots\cup\mathcal{V}_N$, and the total number of the game participants   $n_{\mathrm{sum}}=\sum_{i=1}^Nn_i$.
%Let $i\in\mathcal{I}~(N\in\mathbb{N}^{+})$ be the index of the co
%a well{-}studied class of distributed optimization problems with equality constraints, usually called distributed resource allocation or distributed economic dispatch.
Each agent $ij\in\mathcal{V}$ possesses some local resource, denoted $R_{ij}\in\mathbb{R}$.
For each coalition $i\in\mathcal{I}$, the members are required to cooperatively make the best decision of re{-}allocating the coalition resource, denoted $R_i=\sum_{ij\in\mathcal{V}_i}R_{ij}$, with the goal of achieving maximal coalition{-}level benefit.
%The spatial distribution, privacy requirements, scale, and quantity of information associated with typical networked systems do not allow for centralized communication and decision{-}making, but require instead the use of distributed protocols. That is, the agents in the system must make decisions independently in response to available information.
%In this problem, centralized information cannot be acquired, and each agent can obtain information through a communication.
Denote the decision (state) of agent $ij$ by $x_{ij}\in\mathbb{R}$ and the collective decision (state) of coalition $i$ by $\bm{x}_i=[x_{i1},x_{i2},\cdots,x_{in_i}]^T\in\mathbb{R}^{n_i}$. Define $\bm{x}=[\bm{x}_1^T,\cdots,\bm{x}_N^T]^T\in\mathbb{R}^{n_\mathrm{sum}}$. %For each agent  $ij\in\mathcal{V}$, the payoff, denoted $f_{ij}(\bm{x})$, is contingent on not only its own decision but also other agents' decisions.
If the coalition{-}level benefit of each coalition $i$ is affected by only the collective decision of its members, i.e., $\bm{x}_i$, then, the DRA of the $N$ coalitions can be separated into $N$ independent well{-}studied multi{-}agent distributed optimization problem.
However, in more general competitive situations, the benefit of each coalition may be influenced by  the decisions of not only its members but also the agents outside this coalition, and the conflicts of interests among the coalitions make it necessary to investigate this problem from a game-theoretic perspective. Within this context,
%The coalition{-}level benefit is usually characterized as a function of the decision $\bm{x}$, which is the sum of the objective function of the individual coalition members. cost minimization or welfare maximization.
%Consider the resource allocation problem of multiple coalitions indexed by $i\in\mathcal{I}~(N\in\mathbb{N}^{+})$, where coalition $i$ consists of  $n_i(n_i\in\mathbb{N}^{+})$ agents.  Denote the  the agent set of coalition $i$ by $\mathcal{V}_i=\{i1,i2,\cdots,in_i\}$ with $ij$ representing the $j$th agent in coalition $i$.  Define the agent set of the game $\mathcal{V}=\mathcal{V}_1\cup\cdots\cup\mathcal{V}_N$, and the total number of the game participants   $n_{\mathrm{sum}}=\sum_{i=1}^Nn_i$.  For every $i\in\mathcal{I}$, define the quantity of coalition resources $R_i=\sum_{ij\in\mathcal{V}_i}R_{ij}$ with  $R_{ij}$ representing the quantity of the local resources accessible to agent $ij$, and define the coalition payoff $f_i(\bm{x})=\sum_{ij\in\mathcal{V}_i} f_{ij}(\bm{x})$. The agents belong to the same coalition intend to cooperatively find a best plan to  re{-}allocate coalition resources to maximize the coalition payoff.
%Consider multiple coalitions, each of which contains a group of members that are requested to cooperatively make the best decision of allocating the coalition resource to with the goal of optimizing a coalition{-}level objective.  The informational availability is. we focus on designing distributed protocols that adapt to the underlying interaction topology for settling the agent decisions on the NE.
we formulate the resource allocation game as follows:
\begin{equation}\label{eq.problem1}
\begin{aligned}
&\min_{{x}_{ij}}f_i(\bm{x})=\min_{{x}_{ij}}\sum_{l=1}^{n_i}f_{il}({\bm{x}}),~\forall ij\in\mathcal{V},
\\
&\textit{s.t.}\quad \sum_{ij\in\mathcal{V}_i}{x}_{ij}=R_i,
\end{aligned}
\end{equation}
where $f_{il}:\mathbb{R}^{n_{\mathrm{sum}}}\rightarrow\mathbb{R}$ and $f_i:\mathbb{R}^{n_{\mathrm{sum}}}\rightarrow\mathbb{R}$  are respectively the objective functions of agent $il$ and coalition $i$, and other symbols have been defined previously. Here, we consider the minimization setting without loss of generality, since the case of welfare maximization can be easily transformed into a minimization problem.
The study of this model has potential applications in the fields of economics and engineering.
%In the following,

%\begin{example} (Business Budget Allocation)
\textit{Example 1 (Business Budget Allocation)}:
%Take business budget allocation for example.
Consider that multiple firms, each of which has several product lines, manufacture related products in a competitive market.
%For each product, the output will be influenced by the budget assigned to the product line, and the revenue it generates will be affected by the outputs of itself as well as other homogenous products.
The revenue a product line generates will be influenced by the budgets assigned to the product line well as other homogenous product lines.
Each firm wants to efficiently and effectively use its resource to maximize its total revenue. Such a problem can be modeled as the resource allocation game (\ref{eq.problem1}). $\hfill\blacksquare$
%\end{example}
%consider the budget allocation of  in a competitive market containing multiple firms. Each firm runs several related projects, and the budget one firm allocates to some project will have influence on the profits of this project and others, especially those of the same type running by other firms. contains several subsidiaries, produce related products. Each firm should make a decision on how to allocate budget to the subsidiaries. The budget that each subsidiary get, will influence its output, which further affect the profit of itself as well as the profits of other subsidiaries within and without the firm. it gains, even

%In some situations, each firm only contains heterogeneous product lines, whose revenues do not explicitly depend on the budgets of other product lines of the same firm, but only rely on the budgets of the homogeneous product lines of other firms.

%In model (1), the benefit of each agent may be influenced by all . In some cases, the benefit of each agent is influenced by the decisions of itself and agents in oth

%A class of cases can be characterized by the following model, which may simplify the NE algortihm design
In (\ref{eq.problem1}), the individual agent benefit is expressed by a function of the decisions of all the game participants, i.e., $f_{ij}(\bm{x})$. Note that in some situations, the individual agent benefit can be  formulated as a function of the decisions of only the agent itself and the agents in other coalitions, i.e., $f_{ij}(x_{ij},\bm{x}_{{-}i})$,  where  $\bm{x}_{{-}i}=[\bm{x}_{1}^T,\cdots,\bm{x}_{i{-}1}^T,\bm{x}_{i{+}1}^T,\cdots,\bm{x}_{N}^T]^T$, and the resource allocation game becomes
%For the case that the product lines belong to the same firm are heterogeneous, the competitors of each product line in the market all belong to other firms. Then, the problem can be modeled as
\begin{equation}\label{eq.problem2}
\begin{aligned}
&\min_{{x}_{ij}}f_i(\bm{x})=\min_{{x}_{ij}}\sum_{l=1}^{n_i}f_{il}(x_{il },{\bm{x}_{{-}i}}),~\forall ij\in\mathcal{V},
\\
&\textit{s.t.}\quad \sum_{ij\in\mathcal{V}_i}{x}_{ij}=R_i.
\end{aligned}
\end{equation}
Obviously, model (\ref{eq.problem2}) is a \textit{special case} of (\ref{eq.problem1}), and an example is given as follows.

%and intuitively, the distributed NE computation design for game (\ref{eq.problem2}) will be simpler than that for game (\ref{eq.problem1}). Game (\ref{eq.problem2}) is worthy of investigation, since it characterizes a class of situations, and the following is an example.

%\begin{example} revisited %(Business Budget Allocation)
\textit{Example 1 revisited (Business Budget Allocation)}:
Consider the problem of business budget allocation in Example 1. In each firm, the product lines are heterogeneous, whose revenues do not explicitly depend on the budgets for other product lines in this firm, but only rely on the budgets for themselves and the homogeneous product lines in other firms. Such a problem can be modeled as the resource allocation game (\ref{eq.problem2}). $\hfill\blacksquare$
Next, the definition of NE will be introduced.  Let $(\bm{x}_{i},\bm{x}_{{-}i})\triangleq \bm{x}$
for notational brevity. For each  $i\in\mathcal{I}$,  define the admissible set of the coalition decision  as $\Omega_i=\{\bm{x}_i\in\mathbb{R}^{n_i}|\bm{1}_{n_i}^T\bm{x}_i{=}R_i\}$ and define %$\Omega{=}\{\bm{x}=[\bm{x}_1^T,\bm{x}_2^T,\cdots,\bm{x}_{N}^T]^T\in\mathbb{R}^{n_{\textit{sum}}}|\bm{x}_i{\in}\Omega_i,~i\in\mathcal{I}\}.$
$\Omega{=}\prod_{i=1}^N\Omega_i.$
\begin{definition} An  NE of the resource allocation game (\ref{eq.problem1}) (game (\ref{eq.problem2}))
is a vector $\bm{x}^*=(\bm{x}_{i}^{*},\bm{x}_{{-}i}^*)\in\Omega$ with the property that  $\forall i\in\mathcal{I}$ :
	$$f_i(\bm{x}_{i}^{*},\bm{x}_{{-}i}^*){\leq}f_i(\bm{x}_{i},\bm{x}_{{-}i}^*),~ \forall \bm{x}_i\in\Omega_i.$$
\end{definition}

Define the pseudo gradient function $\mathcal{P}:\mathbb{R}^{n_{\mathrm{sum}}}\rightarrow\mathbb{R}^{n_{\mathrm{sum}}}$ as
$$\mathcal{P}(\cdot)=[(\frac{\partial f_1}{\partial \bm{x}_1}(\cdot))^T,(\frac{\partial f_2}{\partial \bm{x}_2}(\cdot))^T,\cdots,(\frac{\partial f_N}{\partial \bm{x}_N}(\cdot))^T]^T.$$ %$$\mathcal{P}(\cdot)=[(\frac{\partial f_1(\cdot)}{\partial \bm{x}_1})^T,(\frac{\partial f_2(\cdot)}{\partial \bm{x}_2})^T,\cdots,(\frac{\partial f_N(\cdot)}{\partial \bm{x}_N})^T]^T\in\mathbb{R}^{n_{\textit{sum}}}.$$
In this paper, the objective functions of all the agents in game (\ref{eq.problem1}) and (\ref{eq.problem2}) are assumed to satisfy the following assumptions.

\begin{assumption}\label{assp.fij.lipschitz}
	For each agent $ij\in\mathcal{V}$, %$f_{ij}(\bm{x})$
	the  objective function $f_{ij}(\cdot)$ is convex and continuously differentiable. Moreover, $\nabla f_{ij}(\cdot)$ is Lipschitz with the constant $l_{ij}$.% i.e., $\|\nabla f_{ij}(\bm{a}){{-}}\nabla f_{ij}(\bm{b})\|{\leq}l_{ij}\|\bm{a}{-}\bm{b}\|,\forall\bm{a},\bm{b}\in\mathbb{R}^{n_\mathrm{sum}}$.
\end{assumption}
Under Assumption 1, it is not difficult to derive that $\nabla f_{i}(\cdot)$ is Lipschitz with a constant $l_i=\sum_{j=1}^{n_i}l_{ij}$, i.e., $\|\nabla f_{i}(\bm{a}_1){{-}}\nabla f_{i}(\bm{a}_2)\|{\leq}l_{i}\|\bm{a}_1{-}\bm{a}_2\|,\forall\bm{a}_1,\bm{a}_2\in\mathbb{R}^{n_\mathrm{sum}}$, which is useful for the forthcoming  algorithm design and the convergence analysis. %of the proposed algorithms.
%Under Assumption \ref{assp.fij.lipschitz}, we have the following results.
%\begin{equation}
%\begin{aligned}
%&\|\frac{\partial f_{ij}(\bm{\xi}_{ij}(k))}{\partial \bm{x}_i}{{-}}
%\frac{\partial f_{ij}}{\partial \bm{x}_i}(\bm{x}(k))\|\\
%{\leq}&\|\nabla f_{ij}(\bm{\xi}_{ij}(k)){{-}}
%\nabla f_{ij}(\bm{x}(k))\|\\
%{\leq}&l_{ij}\|\bm{\xi}_{ij}(k){{-}}\bm{x}(k)\|
%\end{aligned}
%\end{equation}	

%Define the Laplacian matrix of graph  $\mathcal{G}_i$ by  $\mathcal{L}_i=[l_{ij}^{il}]_{n_i\times n_i}$, where  $l_{ij}^{il}$ denotes the element of $\mathcal{L}_i$ on the $j$th row and $l$th column with $l_{ij}^{ij}=\sum_{l=1}^{n_i}a_{ij}^{il}$ and $l_{ij}^{il}={-}a_{ij}^{il},~j\neq l$.

\begin{assumption}\label{assp.pseudo}(Strictly Monotone Pseudo{-}gradient)
	$(\bm {a}_1{{-}} \bm{a}_2)^T (\mathcal{P}(\bm {a}_1){{-}}\mathcal{P}(\bm {a}_2)){\geq} \mu\|\bm {a}_1{{-}}\bm {a}_2\|^2,\forall\bm{a}_1,\bm{a}_2\in\mathbb{R}^{n_\mathrm{sum}}$, where ${\mu}$ is a positive constant.
\end{assumption}
Assumptions 1 and 2 are quite common in the studies of distributed NE computation \cite{multi-cluster-ymj2018,nianxiaohong2021,Pang2022Auto,mengmin2020,jialing2022arxiv-multi-coalition}, which together ensure the  existence and the uniqueness of  NE in games (\ref{eq.problem1}) and (\ref{eq.problem2}).

%The above assumptions are quite common in the study of NE seeking \cite{Pavel_Auto2016}, which together ensure the  existence and the uniqueness of  NE.%in game (\ref{eq.problem2}) and thus in game (\ref{eq.problem1}). %

%\begin{remark}
%Since
%	\begin{equation}\label{eq.proper.y*.g}
%	\frac{\partial g_i}{\partial {y}_{i}}(\bm{y}^*)=0,\quad \forall i\in\mathcal{I},
%	\end{equation}
%	and
%	\begin{equation}
%	\begin{aligned}
%	\frac{\partial g_i}{\partial {y}_{i}}(\bm{y})=&\sum_{j=1}^{n_i}\frac{\partial f_i}{\partial x_{ij}}([y_1\bm{1}_{n_1}^T,y_2\bm{1}_{n_2}^T,\cdots,y_N\bm{1}_{n_N}^T]^T)\\
%	&=\bm{1}_{n_i}^T\frac{\partial f_i}{\partial \bm{x}_{i}}([y_1\bm{1}_{n_1}^T,y_2\bm{1}_{n_2}^T,\cdots,y_N\bm{1}_{n_N}^T]^T),%\in\mathbb{R}
%	\end{aligned}
%	\end{equation}
%	one has
%	\begin{equation}\label{eq.proper.y*.f}
%	\begin{aligned}
%	&\bm{1}_{n_i}^T	\frac{\partial f_i}{\partial \bm{x}_{i}}([y_1^*\bm{1}_{n_1}^T,y_2^*\bm{1}_{n_2}^T,\cdots,y_N^*\bm{1}_{n_N}^T]^T)\\
%	=&\bm{1}_{n_i}^T\frac{\partial f_i}{\partial \bm{x}_{i}}(\bm{x}^*)=0,~~\forall i\in\mathcal{I},
%	\end{aligned}
%	\end{equation}
%	which is in consistence with the analysis in Remark 1. 	$\hfill\square$
%	%Obviously, $\mathcal{Q}(\bm{y}^*)=0$.
%\end{remark}	

%This paper considers the resource allocation game over network, where the objective function and the decision of each agent are private to itself, while some agents can exchange information with each other.
%The agents cannot directly calculate the NE since they can not obtain global information

\subsection{Network Topology and the Associated Matrices}\label{sec.topology and matices}

The underlying network topology among the $n_{\mathrm{sum}}$ game participants is depicted by an  undirected graph $\mathcal{G}(\mathcal{V},\mathcal{E})$ with $\mathcal{V}$ and $\mathcal{E}\subseteq \mathcal{V}\times \mathcal{V}$ respectively denoting  the node (agent) set and the edge (communication link) set. A pair $(ij,pq)\in\mathcal{E}$ is an edge of $\mathcal{G}$ if agent $pq$ can receive information from agent $ij$.  If  $(ij,pq)\in\mathcal{E}$, then agent $ij$ is called a \textit{neighbor} of agent $pq$. The graph $\mathcal{G}$ is assumed to be undirected, i.e., for any $(ij,pq)\in\mathcal{E}$, $(pq,ij)\in\mathcal{E}$.
A  path from agent $i_1 j_1$ to agent $i_l j_l$ is
a sequence of edges $(i_mj_m,i_{m{+}1}j_{m{+}1})\in\mathcal{E},m=1,\cdots,l{-}1$.
The undirected graph $\mathcal{G}$ is called connected if for any agent, there exist paths to all
other agents.
Define  the induced subgraph  $\mathcal{G}_i(\mathcal{V}_i,\mathcal{E}_i)$ with the node set $\mathcal{V}_i$ and the edge set $\mathcal{E}_i=\{(ij,il)|(ij,il)\in\mathcal{E}\}\subseteq \mathcal{V}_i\times \mathcal{V}_i$. Obviously, $\mathcal{G}_i$ characterizes the underlying topology among the agents in coalition $i$.
For each agent $ij\in\mathcal{V}$, define the neighbor set $\mathcal{N}_{ij}=\{lm|(lm,ij)\in\mathcal{E}\}$,  the intra{-}coalition neighbor set $\mathcal{N}_{ij}^{i}=\{il|(il,ij)\in\mathcal{E}_i\}$, the degree $d_{ij}=|\mathcal{N}_{ij}|$ and the intra-coalition degree $d_{ij}^{i}=|\mathcal{N}_{ij}^i|$.

The adjacency matrix of graph $\mathcal{G}$ is defined as $A=[a_{ij}^{pq}]_{n_{\mathrm{sum}}\times n_{\mathrm{sum}}}$  with $a_{ij}^{ij}=0$, and
$a_{ij}^{pq}=1$ if $(pq,ij)\in\mathcal{E}$ and $0$ otherwise, where $a_{ij}^{pq}$ denotes the element of $A$ on the $(\sum_{k=1}^{i{-}1}n_k{+}j)$th row and the $(\sum_{k=1}^{p{-}1}n_k{+}q)$th column. Similarly, the adjacency matrix of the subgraph  $\mathcal{G}_i~(i\in\mathcal{I})$ is defined by $A_i=[a_{ij}^{il}]_{{n_i}\times n_i}$ with $a_{ij}^{il}$ denoting the $(j,l)${-}entry of $A_i$. Obviously,  $A_1,\cdots,A_N$ are the diagonal blocks of $A$.
The Laplacian matrix of  $\mathcal{G}_i$ is defined as $L_i=[l_{ij}^{il}]_{n_i\times n_i}$ with
$l_{ij}^{ij}=\sum_{l=1}^{n_i}a_{ij}^{il}$ and $l_{ij}^{il}={-}a_{ij}^{il},~j\neq l$, where $l_{ij}^{il}$ denotes the   $(j,l)${-}entry of $L_i$.
%For each coalition $i$, define a weighted adjacency matrix of the graph $\mathcal{G}_i$ is further defined as  $C=[c_{ij}^{im}]_{n_{i}\times n_{i}}$  with $c_{ij}^{ij}=0$, and $c_{ij}^{pq}>0$

Apart from the above matrices, a weighted adjacency matrix of the graph $\mathcal{G}$ is further defined as  $W=[w_{ij}^{pq}]_{n_{\mathrm{sum}}\times n_{\mathrm{sum}}}$  with $w_{ij}^{ij}=0$,
$w_{ij}^{pq}>0$ if $(pq,ij)\in\mathcal{E}$ and $0$ otherwise, and $\forall ij\in\mathcal{V}$, $\sum_{pq\in\mathcal{V}} w_{ij}^{pq} {+}\max_{pq\in\mathcal{V}} \{w_{ij}^{pq}\}<1  $ (similar to the  superscripts and subscripts in entries of the adjacency matrix  $A$, $w_{ij}^{pq}$ denotes the element of $W$ on the $(\sum_{k=1}^{i{-}1}n_k{+}j)$th row and the $(\sum_{k=1}^{p{-}1}n_k{+}q)$th column).  For example, the entries $w_{ij}^{pq}~\forall ij,pq\in\mathcal{V}$ can be set as
$
w_{ij}^{pq}={a_{ij}^{pq}}/{h_{ij}}~$, where $h_{ij}>d_{ij}+\max_{pq\in\mathcal{V}}\{a_{ij}^{pq}\}$ is a constant.
%Furthermore,  the entries of $W$ is designed to satisfy $\forall ij\in\mathcal{V}$, $\sum_{pq\in\mathcal{V}} w_{ij}^{pq} {+}\max_{pq\in\mathcal{V}} \{w_{ij}^{pq}\}<1  $ in this paper.

For each coalition $i$, a doubly-stochastic  matrix associated with graph $\mathcal{G}_i$ is defined as  $C_i=[c_{ij}^{im}]_{n_{i}\times n_{i}}$  with  $c_{ij}^{im}>0$ if $im\in\mathcal{N}_{ij}^i\cup \{ij\}$ and $c_{ij}^{im}=0$ otherwise. For example,
the entries of $C_i$ can be set as
$c_{ij}^{ij}=1-{d_{ij}^i}/{n_i}$ and $
c_{ij}^{im}={a_{ij}^{im}}/{n_i},\forall im\neq ij$.
%Another choice is  $c_{ij}^{ij}=1/(1+d_{ij}^i)$ and $c_{ij}^{im}={a_{ij}^{im}}/(1+d_{im}^i),\forall im\neq ij$.

\begin{assumption}\label{assp.graph}
	The graph $\mathcal{G}$ is undirected and connected, and all the sub{-}graphs $\mathcal{G}_i(i\in\mathcal{I})$ are undirected and connected.
\end{assumption}

\subsection{Design Objective}
In the previous subsections,
the resource allocation problem over multiple interacting coalitions has been formulated
as a multi-coalition game, % that captures the conflicts of interests among coalitions,
and the communication topology among the individual agents in these coalitions have also been described. Specifically,
each individual agent $ij\in\mathcal{V}$ is aware of its own decision value $x_{ij}$ and objective function $f_{ij}$, and it can share the local information with its neighbors through the communication network.
%and besides, it can obtain the information of its in{-}neighbours  via  the directed communication network $\mathcal{G}$.

Next, % we will design
distributed NE seeking algorithms will be developed for the  proposed resource allocation games  (\ref{eq.problem1}) and (\ref{eq.problem2}) (i.e., \textit{the general and special cases}). The design objective is to make the collective agent decision $\bm{x}$ converge to the NE $\bm{x}^*$ of the proposed games, with the information utilization adapting to  the network topology $\mathcal{G}$.
%the  objective is to develop distributed NE seeking algorithms to make the  collective agent decision $\bm{x}$ converge to the NE $\bm{x}^*$ of the proposed resource allocation games  (\ref{eq.problem1}) and (\ref{eq.problem2}) (i.e., \textit{the general and special cases}), with the information utilization adapting to  the network topology $\mathcal{G}$.
% under which the collective agent decision $\bm{x}$ will converge to the NE $\bm{x}^*$ of the proposed resource allocation games. %$\bm{x}(0)$.

To proceed, we illustrate  a property of the NE of game (\ref{eq.problem1}) (game (\ref{eq.problem2})) in the following lemma, which is helpful for the NE seeking algorithm design.
%Before proceeding,  a property of the NE of game (\ref{eq.problem1}) (game (\ref{eq.problem2})) is illustrated in the following lemma, which is helpful for the NE seeking algorithm design.
%\begin{remark}
%	Using the method of Lagrange multipliers, define the Hamilton function for each coalition
%	\begin{equation}
%	H_i(\bm{x}_i,\lambda_i)=f_i(\bm{x}){+}\lambda_i(\bm{1}_{n_i}^T\bm{x}_i{-}R_i),
%	\end{equation}
%	where $\lambda_i\in\mathbb{R}$ is the Lagrange multiplier. Then,
%	one can derive that the NE of (\ref{eq.problem1}) and (\ref{eq.problem2}) satisfy $$\frac{\partial H_i}{\partial \bm{x}_{i}}(\bm{x}^*,\lambda_i)=\frac{\partial f_i}{\partial \bm{x}_{i}}(\bm{x}^*){+}{\lambda}_i\bm{1}_{n_i}=0,~\forall i\in\mathcal{I}.$$
%	% where $\bm{\lambda}_i^*,i\in\mathcal{I}$ are the optimal Lagrange multipliers.
%	This further implies that
%	$\forall i\in\mathcal{I}$, $$L_i\frac{\partial f_i}{\partial \bm{x}_i}(\bm{x}^*)=0,$$ or $$\frac{\partial f_i}{\partial {x}_{i1}}(\bm{x}^*)=\frac{\partial f_i}{\partial {x}_{i2}}(\bm{x}^*)=\cdots=\frac{\partial f_i}{\partial {x}_{in_i}}(\bm{x}^*).$$$\hfill\square$
%\end{remark}

\begin{lemma}\label{lemma.NE}
	Suppose Assumptions 1-3 hold. A vector
	$\bm{x}^*$ is the NE of game (\ref{eq.problem1}) (game (\ref{eq.problem2})) if and only if $\bm{x}^*$  satisfies
	$$L_i\frac{\partial f_i}{\partial \bm{x}_i}(\bm{x}^*)=0,	~\forall i\in\mathcal{I}.$$
\end{lemma}
\begin{proof}
	Define the Hamilton function for each coalition
	%	\begin{equation*}
	$	H_i(\bm{x}_i,\lambda_i)=f_i(\bm{x}){+}\lambda_i(\bm{1}_{n_i}^T\bm{x}_i{-}R_i)$,
	%	\end{equation*}
	where $\lambda_i\in\mathbb{R}$ is the Lagrange multiplier. Note that Assumptions 1-3 hold. Then, from the well-known Karush-Kuhn-Tucker (KKT) optimality condition \cite{KKTbook}, one can obtain that, $\bm{x}^*$ is the NE of game (\ref{eq.problem1}) (game (\ref{eq.problem2})) if and only if there exist a ${\lambda}_i$ that satisfies
	%	one can derive that the NE of (\ref{eq.problem1}) and (\ref{eq.problem2}) satisfy
	$\frac{\partial H_i}{\partial \bm{x}_{i}}(\bm{x}^*,\lambda_i)=\frac{\partial f_i}{\partial \bm{x}_{i}}(\bm{x}^*){+}{\lambda}_i\bm{1}_{n_i}=0,~\forall i\in\mathcal{I},$ which is further equivalent to  $L_i\frac{\partial f_i}{\partial \bm{x}_i}(\bm{x}^*)=0,	~\forall i\in\mathcal{I}.$
	%	% where $\bm{\lambda}_i^*,i\in\mathcal{I}$ are the optimal Lagrange multipliers.
	%	This further implies that
	%	$\forall i\in\mathcal{I}$, $$L_i\frac{\partial f_i}{\partial \bm{x}_i}(\bm{x}^*)=0,$$ or $$\frac{\partial f_i}{\partial {x}_{i1}}(\bm{x}^*)=\frac{\partial f_i}{\partial {x}_{i2}}(\bm{x}^*)=\cdots=\frac{\partial f_i}{\partial {x}_{in_i}}(\bm{x}^*).$$$\hfill\square$
\end{proof}

%%%%%%%%%%%%%%%%%%%%%%%%%%%%%%%%%%%%%%%%%%%%%%%%%%%%%%%%%%%%%%%%%%%%%%%%%%%%%%%%%%%%%%%%%%
\section{Distributed NE computation for the special case}\label{sec.special case}
%\section{Distributed NE computation in the special resource allocation game}
%\section{Distributed NE computation in game (2)}
%Noting that (2) is a special case of 1, where less relevance are involoved in
Intuitively, the distributed NE computation design for \textit{the special case} described by model (\ref{eq.problem2}) is simpler than that for \textit{the general case} described by model (\ref{eq.problem1}), since less relevance of the agents' decisions to the individual objectives are involved in the former.
%Therefore, a distributed NE computation algorithm for \textit{the special case} of the proposed resource allocation game describe by model (\ref{eq.problem2}) will be first developed in this section. Then, a distributed NE computation algorithm for \textit{the general case}  of the proposed resource allocation game (\ref{eq.problem1}) will be further designed in the next section.
Therefore, we will first study the distributed NE computation for \textit{the special case} of the proposed resource allocation game (model (\ref{eq.problem2})). Based on the results in this section, the issue for \textit{the general case} of the proposed resource allocation game (model (\ref{eq.problem1})) will be investigated in the next section.

\subsection{Algorithm Design}
To make the collective agent state converge to  %distributively
the NE in game (\ref{eq.problem2}), the following distributed algorithm is designed for each agent $ ij\in\mathcal{V}$ $\forall k\in\mathbb{N}$:
\begin{subequations}\label{eq.law.1}
\begin{align}
x_{ij}(k)=&x_{ij}(0){{-}}\sum_{im\in\mathcal{N}_{ij}^{i}} (\eta_{ij}(k){{-}}\eta_{im}(k)),\label{eq.law.x.ij}\\
{\eta}_{ij}(k{{+}}1)=&{\eta}_{ij}(k){+}{\alpha}\sum_{im\in\mathcal{N}_{ij}^{i}} \bigg(\frac{\partial f_{ij}(\xi_{ij}^{ij}(k),\bm {\xi}_{ij}^{{-}i}(k))}{\partial {x}_{ij}}\nonumber\\
&{-}\frac{\partial f_{im}(\xi_{im}^{im}(k),\bm{\xi}_{im}^{{-}i}(k))}{\partial {x}_{im}}\bigg),\label{eq.law.eta.ij}\\
\xi_{ij}^{pq}(k{{+}}1)=&\bar{w}_{ij}^{pq}\xi_{ij}^ {pq}(k){{+}}\sum\limits_{lm\in\mathcal{N}_{ij}}w_{ij}^{lm}\xi_{lm}^{pq}(k)\nonumber\\
&{+}w_{ij}^{pq}x_{pq}(k),~~~\forall pq\in\mathcal{V},\label{eq.law.xi.ijpq}
\end{align}
\end{subequations}
where $x_{ij}(0)=R_{ij}$, $\eta_{ij}(0)=0$, $\alpha$ is a small positive constant to be determined, $\xi_{ij}^{pq}$ is an auxiliary variable computed by agent $ij$ for estimating the value of $x_{pq}$,  $\bm{\xi}_{ij}^l=[\bm{\xi}_{ij}^{l1},\cdots,\bm{\xi}_{ij}^{ln_l}]^T,\forall l\in\mathcal{I}$,  $\bm{\xi}_{ij}^{{-}i}=[(\bm{\xi}_{ij}^1)^T,\cdots,(\bm{\xi}_{ij}^{i{-}1})^T,(\bm{\xi}_{ij}^{i{+}1})^T,\cdots,(\bm{\xi}_{ij}^{N})^T]^T$, and $\bar{w}_{ij}^{pq}=1{-}\sum_{lm\in\mathcal{N}_{ij}}w_{ij}^{lm}{-}{w}_{ij}^{pq}$. % Since the entries of $W$ is designed to satisfy $\forall ij\in\mathcal{V}$, $\sum_{pq\in\mathcal{V}} w_{ij}^{pq} {+}\max_{pq\in\mathcal{V}} \{w_{ij}^{pq}\}<1 $, one has $\bar{w}_{ij}^{pq}>0,\forall ij,pq\in\mathcal{V}$.
Note from the definition of parameters $w_{ij}^{pq},\forall ij,pq\in\mathcal{V}$ in Sec. \ref{sec.topology and matices} that $\bar{w}_{ij}^{pq}>0,\forall ij,pq\in\mathcal{V}$.

To rewrite the proposed algorithm (\ref{eq.law.1}) in a compact form, we
define the following vectors
\begin{equation*}
\begin{aligned}
&\bm{\eta}_{i}=[{\eta}_{i1},\cdots,{\eta}_{in_i}]^T\in\mathbb{R}^{n_i},\\
&\bm{\eta}=[(\bm{\eta}_{1})^T,(\bm{\eta}_{2})^T,\cdots,(\bm{\eta}_{N})^T]^T\in\mathbb{R}^{n_{\mathrm{sum}}},\\
&\bm{\xi}_{ij}=[\xi_{ij}^{11},\xi_{ij}^{12},{\cdots},\xi_{ij}^{1{n_1}},\xi_{ij}^{21},{\cdots},\xi_{ij}^{2{n_2}},\cdots,\xi_{ij}^{N{n_N}}]\in\mathbb{R}^{n_{\mathrm{sum}}},\\
&\bm{\xi}_i=[(\bm{\xi}_{i1})^T,(\bm{\xi}_{i2})^T,\cdots,(\bm{\xi}_{i{n_i}})^T]^T\in\mathbb{R}^{n_in_{\mathrm{sum}}},\\
&\bm{\xi}=[(\bm{\xi}_{1})^T,(\bm{\xi}_{2})^T,\cdots,(\bm{\xi}_{N})^T]^T\in\mathbb{R}^{n_{\mathrm{sum}}^2},
\end{aligned}
\end{equation*}
and
the function $\breve{\mathcal{P}}_i:\mathbb{R}^{n_i^2}\rightarrow\mathbb{R}^{n_i}$:
\begin{equation*}
\begin{aligned}
\breve{\mathcal{P}}_i(\bm{\xi}_i)%=&[\frac{\partial f_{i1}}{\partial x_{i1}}(\bm{\xi}_{i1}),\frac{\partial f_{i2}}{\partial x_{i2}}(\bm{\xi}_{i2}),{\cdots},\frac{\partial f_{in_i}}{\partial x_{i{n_i}}}(\bm{\xi}_{in_i})]^T\\
=&[\frac{\partial f_{i}}{\partial x_{i1}}(\bm{\xi}_{i1}),\frac{\partial f_{i}}{\partial x_{i2}}(\bm{\xi}_{i2}),{\cdots},\frac{\partial f_{i}}{\partial x_{i{n_i}}}(\bm{\xi}_{in_i})]^T\in\mathbb{R}^{n_i}.
\end{aligned}
\end{equation*}
Obviously, $\breve{\mathcal{P}}_i(\bm{1}_{n_i}{\otimes} \bm{x})=\frac{\partial f_i}{\partial \bm{x}_i}(\bm{x})$.
Note that for each agent  in game (\ref{eq.problem2}), the inputs of individual objective function include the decisions of only itself and the agents in other coalitions. Therefore,
%the extra{-}coalition agents,
one has $\forall ij\in\mathcal{V}, \forall \bm{x}\in\mathbb{R}^{n_\mathrm{sum}}$
\begin{equation}\label{relation}
\frac{\partial f_i}{\partial x_{ij}}(\bm{x})=\sum_{l=1}^{n_i}\frac{\partial f_{il}}{\partial x_{ij}}(x_{il},\bm{x}_{{-}i})=\frac{\partial f_{ij}}{\partial x_{ij}}(x_{ij},\bm{x}_{{-}i}).
\end{equation}
Then, the proposed algorithm (\ref{eq.law.1})
%(\ref{eq.law.1})
can be  rewritten in the following compact form $\forall i\in\mathcal{I}$:
\begin{subequations}\label{eq.law.1.compact}
\begin{align}
\bm{x}_i(t)=&\bm{x}_i(0){{-}}L_i\bm{\eta}_i(k),\label{eq.law.x.i}\\
{\bm{\eta}}_i(k{{+}}1)=&{\bm{\eta}}_i(k){+}\alpha L_i
\breve{\mathcal{P}}_i(\bm{\xi}_i({k})),\label{eq.law.eta.i}\\
\bm{\xi}(k{{+}}1)
=&(W{\otimes} I_{n_\mathrm{sum}}{+}{\bar{W}})\bm{\xi}(k){{+}}\hat{W}\big(\bm{1}_{n_\mathrm{sum}}{{\otimes}}\bm{x}(k)\big),\label{eq.law.xi}
%\bm{\xi}(k{{+}}1)
%=&\bm{\xi}(k){{-}}\Gamma\bigg((L{{\otimes}} I_{n_\mathrm{sum}})\bm{\xi}(k){+}A_d\big(\bm{\xi}(k){{-}}\bm{1}_{n_\mathrm{sum}}{{\otimes}}\bm{x}(k)\big)\bigg)\\
%=&\bm{\xi}(k){{-}}\Gamma(L{{\otimes}} I_{n_\mathrm{sum}}{+}A_d)\big(\bm{\xi}(k)\nonumber\\
%&{{-}}\bm{1}_{n_\mathrm{sum}}{{\otimes}}\bm{x}(k)\big)\label{eq.law.xi}
\end{align}
\end{subequations}
where $W$ has been defined  in Sec. \ref{sec.topology and matices},
%$W$ is the matrix with $w_{ij}^{pq}$ being the element on the on the $(\sum_{k=1}^{i{-}1}n_k{+}j)$th row and the $(\sum_{k=1}^{p{-}1}n_k{+}q)$th column,
${\bar{W}}=\mathrm{diag}\{\bar{w}_{11}^{11},\cdots,\bar{w}_{11}^{Nn_N},\bar{w}_{12}^{11},\cdots,\bar{w}_{12}^{Nn_N},\cdots,\bar{w}_{Nn_N}^{Nn_N}\}$,
$\hat{W}=\mathrm{diag}\{w_{11}^{11},\cdots,w_{11}^{Nn_N},w_{12}^{11},\cdots,w_{12}^{Nn_N},\cdots,w_{Nn_N}^{Nn_N}\}$, and (\ref{eq.law.eta.i}) is obtained by using (\ref{relation}).
%It is not difficult to verify that $(W{\otimes} I_{n_\mathrm{sum}}{+}{\bar{W}})$ is a Schur matrix.

\begin{remark}
	%The design of (\ref{eq.law.x.ij}) is inspired by the distributed optimization algorithm for resource allocation of single MAN in \cite{ChenLiAuto2018ConDis}.
	In the forthcoming convergence analysis of the proposed algorithm, we will show that
	(\ref{eq.law.x.ij}) ensures the satisfaction of the  equality constraint during the whole process of NE seeking. The design of (\ref{eq.law.x.ij}) is  inspired by the distributed optimization algorithms for the resource allocation over a single coalition in some existing literature, e.g., \cite{ChenLiAuto2018ConDis,zjl-dispath-tnse}. Under (\ref{eq.law.x.ij}), the primal problem of finding the NE regrading the objective functions of the variable $\bm{x}$ subject to the equality constraints, can be converted to a problem regarding the composite functions of the variable $\bm{\eta}$ without any constraint. Noticing this, after the variable replacement of $\bm{x}$ by $\bm{\eta}$, (\ref{eq.law.eta.ij}) can be viewed as a pseudo-gradient descent law for the equivalent problem, with the collective state $\bm{x}$ estimated by the leader-following consensus protocol (\ref{eq.law.xi.ijpq}). The consensus-based state estimation is quite common in distributed NE seeking,  since the collective state is required in the iteration of each agent, while the information acquisition is subject to the communication topology $\mathcal{G}$. If there is only one coalition in the problem (\ref{eq.problem2}), then, state estimation (\ref{eq.law.xi.ijpq}) is no longer required,
	and the proposed algorithm will degenerate into a single-coalition DRA algorithm with linear convergence, which has a similar structure to the DRA algorithms in \cite{ChenLiAuto2018ConDis,zjl-dispath-tnse}.	
	%and the proposed algorithm will degenerate into a DRA algorithm that converge linearly to the optimal solution of the DRA problem over a single coalition, which has a similar structure to the DRA algorithms in \cite{ChenLiAuto2018ConDis,zjl-dispath-tnse}.
	 %In this sense, the methodology developed in this paper generalizes the existing results on DRA over MANs.

\end{remark}

\subsection{Convergence Analysis}

From (\ref{eq.law.x.i}) and (\ref{eq.law.eta.i}), one can get
\begin{equation}\label{eq.x.i.k{+}1.k}
\bm{x}_i(k{{+}}1){-}\bm{x}_i(k)={{-}}\alpha L_i^2\breve{\mathcal{P}}_i(\bm{\xi}_i({k})),
\end{equation}
which can be further rewritten as
\begin{equation}\label{eq.x.k{+}1.k}
\bm{x}(k{{+}}1){-}\bm{x}(k)={{-}}\alpha\hat{L}^2\breve{\mathcal{P}}(\bm{\xi}({k})),
\end{equation}
where $\hat{L}=\mathrm{diag}\{L_1,\cdots,L_N\}$ and $\breve{\mathcal{P}}(\bm{\xi})=[\breve{\mathcal{P}}_1^T(\bm{\xi}_1),\cdots,\breve{\mathcal{P}}_N^T(\bm{\xi}_N)]^T$.

%\bm{e_{\xi_i}}{=}\bm{\xi}_i{{-}}\bm{1}_{n_i}{{\otimes}} \bm{x},~~
Define the estimation errors
\begin{equation}\label{eq.exi.def}
\bm{e_\xi}{=}\bm{\xi}{{-}}\bm{1}_{n_\mathrm{sum}}{{\otimes}} \bm{x}.
\end{equation}
Combining  (\ref{eq.law.xi}), (\ref{eq.x.k{+}1.k}) and (\ref{eq.exi.def}),
%and noticing that $(W{\otimes} I_{n_\mathrm{sum}}{+}{\bar{W}}{+}\hat{W})(\bm{1}_{n_\mathrm{sum}}{{\otimes}} \bm{x})=\bm{1}_{n_\mathrm{sum}}{{\otimes}} \bm{x}$,
one can derive that
\begin{equation}\label{eq.e_xi.k{+}1.k}
\begin{aligned}
&\bm{e_\xi}(k{{+}}1)\\
%{=}&\bm{\xi}(k{{+}}1){{-}}\bm{1}_{n_{\mathrm{sum}}}{{\otimes}}   \bm{x}(k{{+}}1)\\
{=}&(W{\otimes} I_{n_\mathrm{sum}}{+}{\bar{W}})\bm{\xi}(k){{+}}\hat{W}\big(\bm{1}_{n_\mathrm{sum}}{{\otimes}}\bm{x}(k)\big){{-}}\bm{1}_{n_{\mathrm{sum}}}{{\otimes}}\bm{x}(k{{+}}1) \\
{=}&(W{\otimes} I_{n_\mathrm{sum}}{+}{\bar{W}})(\bm{\xi}(k){-}\bm{1}_{n_\mathrm{sum}}{\otimes} \bm{x}(k))\\
&{{+}}(W{\otimes} I_{n_\mathrm{sum}}{+}{\bar{W}}{+}\hat{W})\big(\bm{1}_{n_\mathrm{sum}}{{\otimes}}\bm{x}(k)\big){{-}}\bm{1}_{n_{\mathrm{sum}}}{{\otimes}}\bm{x}(k{{+}}1) \\
%{=}&(W{\otimes} I_{n_\mathrm{sum}}{+}{\bar{W}})(\bm{\xi}(k){-}\bm{1}_{n_\mathrm{sum}}{\otimes} \bm{x}(k))\\
%&{{+}}\big(\bm{1}_{n_\mathrm{sum}}{{\otimes}}\bm{x}(k)\big){{-}}\bm{1}_{n_{\mathrm{sum}}}{{\otimes}}\bm{x}(k{{+}}1) \\
{=}&(W{\otimes} I_{n_\mathrm{sum}}{+}{\bar{W}})\bm{e_\xi}(k){{-}}\bm{1}_{n_\mathrm{sum}}{{\otimes}}(\bm{x}(k{{+}}1){-}\bm{x}(k))\\
{=}&\mathcal{M}\bm{e_\xi}(k){{+}}\bm{1}_{n_\mathrm{sum}}{{\otimes}}(\alpha\hat{L}^2\breve{\mathcal{P}}(\bm{\xi}({k}))),
\end{aligned}
\end{equation}
where $\mathcal{M}{=}W{\otimes} I_{n_\mathrm{sum}}{+}{\bar{W}}$,
and  the third equality is obtained by using the fact that $(W{\otimes} I_{n_\mathrm{sum}}{+}{\bar{W}}{+}\hat{W})(\bm{1}_{n_\mathrm{sum}}{{\otimes}} \bm{x})=\bm{1}_{n_\mathrm{sum}}{{\otimes}} \bm{x}$.
Since the graph $\mathcal{G}$ is connected, it is easy to verify from Gershgorin's circle
theorem that $\mathcal{M}$ is a Schur matrix. Therefore, there exist a symmetric
positive definite matrices $W_\mathcal{M}$ such that
$\mathcal{M}^T W_\mathcal{M}\mathcal{M} {{-}}W_\mathcal{M} {=} {{-}}  I_{n_{\mathrm{sum}}^2}$.

\begin{theorem}\label{theo.1}
	Suppose that Assumptions  \ref{assp.fij.lipschitz}{-}\ref{assp.graph}
	%\ref{assp.fij.lipschitz},  \ref{assp.pseudo}, and \ref{assp.graph}
	hold. %and graph $\mathcal{G}$ is undirected and connected. Then,
	Under the proposed DRA algorithm  (\ref{eq.law.1}), %,(\ref{eq.law.x.ij}){-}(\ref{eq.law.xi.ijpq}),
	%(\ref{eq.law.1})
	the collective agent state $\bm{x}$ will converge linearly to the NE of the resource allocation game (\ref{eq.problem2}), if $\alpha$ satisfies
	\begin{equation}\label{alpha}
	\begin{aligned}
	\alpha{\leq}\min\bigg\{\frac{\gamma}{8\mu\max\limits_{i\in\mathcal{I}}\{l_i^2\|{L_i}\|^4\}},\frac{\mu}{2\sum_{i=1}^N(l_i^2\|L_i\|^2){{+}}\gamma b}\bigg\},
	\end{aligned}
	\end{equation}	
	where \begin{equation*}
	\gamma=4\max_{i\in\mathcal{I}}\{l_i^2\|L_i\|^2\}, ~b={n_\mathrm{sum}}(2\|\mathcal{M}^T W_\mathcal{M}\|^2{{+}}\|W_\mathcal{M}\|).\end{equation*}
	Moreover, the equality constraint in (\ref{eq.problem2}) is always satisfied during the iterations.
\end{theorem}
Before presenting the proof of Theorem 1, we first introduce some useful lemmas.

\begin{lemma}\label{lemma.Vxi}
	Under Assumption \ref{assp.graph} and  the proposed algorithm (\ref{eq.law.1}), for the function
	\begin{equation}\label{eq.V.xi}
	V_{\xi}(k){=}\bm{e_\xi}(k)^TW_{\mathcal{M}}\bm{e_{\xi}}(k),
	\end{equation}
	the following inequality holds $\forall k\in\mathbb{N}$:
	\begin{equation*}\label{eq.V.xi.k{+}1.k}
	\begin{aligned}
	V_{\xi}({k{+}1}){{-}}V_{\xi}(k)
	\leq{{-}}\frac{1}{2}\|\bm{e_\xi}(k)\|^2{{+}}\alpha^2b\|\hat{L}^2\breve{\mathcal{P}}(\bm{\xi}({k}))\|^2.
	\end{aligned}
	\end{equation*}
\end{lemma}
The proof of Lemma \ref{lemma.Vxi} is reported in Appendix \ref{proof.lemma.Vxi}.

\begin{lemma}\label{lemma.Vx}
	Under Assumptions \ref{assp.fij.lipschitz}-\ref{assp.graph} and  the proposed algorithm (\ref{eq.law.1}), for the function
	\begin{equation*}
	V_x(k)=\sum_{i=1}^N\left\|L_i\frac{\partial f_i}{\partial \bm{x}_i}(\bm{x}(k))\right\|^2,
	\end{equation*}
	the following inequality holds $\forall k\in\mathbb{N}$:
	\begin{equation*}\label{eq.Vx.k{+}1.k}
	\begin{aligned}
	&V_x(k{{+}}1){-}V_x(k)\\
	\leq&{{-}}2\Big({\mu}{{-}}{\alpha}\sum_{i=1}^N\big(l_i^2\|L_i\|^2\big)\Big)\alpha\left\|\hat{L}^2\breve{\mathcal{P}}(\bm{\xi}({k}))\right\|^2{{+}}\frac{\gamma}{4}\left\|\bm{e_\xi}(k)\right\|^2.
	\end{aligned}
	\end{equation*}
\end{lemma}
The proof of Lemma \ref{lemma.Vx} is shown in Appendix \ref{proof.lemma.Vx}.

Now, we are ready to present the proof of Theorem \ref{theo.1}.\\
\textbf{Proof of Theorem 1:}	

From (\ref{eq.law.x.i}), one has $\mathbf{1}_{n_i}^T\bm{x}_i(k){=}\mathbf{1}_{n_i}^T\bm{x}_i(0){=}R_i,$ $\forall k{\in}\mathbb{N}$, meaning that the equality constraint in problem \dref{eq.problem2} is  satisfied at each iteration under the proposed algorithm.

Consider the following Lyapunov function:
\begin{equation}\label{eq.V}
V(k)=V_x(k){+}\gamma V_\xi(k),
\end{equation}
where $V_\xi(k)$, $V_x(k)$, and  $\gamma$ have been defined in Lemmas \ref{lemma.Vxi}, \ref{lemma.Vx} and Theorem \ref{theo.1} respectively.
From Lemmas \ref{lemma.Vxi} and \ref{lemma.Vx}, one has
%Combining (\ref{eq.Vx.k{+}1.k}) and (\ref{eq.V.xi.k{+}1.k}) yields
\begin{equation*}
\begin{aligned}
&V({k{+}1}){{-}}V(k)\\
%\leq&{-}2({\mu}{-}\alpha\sum_{i=1}^Nl_i^2\|L_i\|^2)\alpha\left\|\hat{L}^2\mathcal{Q}(\bm{\xi}({k}))\right\|^2\\
%&{+}\min_{i\in\mathcal{I}}\{l_i^2\|L_i\|^2\}\left\|\bm{e_\xi}(k)\right\|^2\\
%&{{-}}\frac{\gamma}{2}\|\bm{e_\xi}(k)\|^2{{+}}\gamma{b}\alpha^2\|(\hat{L}^2\mathcal{Q}(\bm{\xi}({k})))\|^2\\
\leq&{{-}}\frac{\gamma}{4}\|\bm{e_\xi}(k)\|^2
{-}\bigg(2{\mu}{-}\alpha\Big(2\sum_{i=1}^N(l_i^2\|L_i\|^2){{+}}\gamma{b}\Big)\bigg)\\
&\times\alpha\left\|\hat{L}^2\breve{\mathcal{P}}(\bm{\xi}({k}))\right\|^2.
%\leq&{{-}}\frac{\gamma}{4}\|\bm{e_\xi}(k)\|^2{-}\mu\alpha\left\|\hat{L}^2\breve{\mathcal{P}}(\bm{\xi}({k}))\right\|^2,
\end{aligned}
\end{equation*}
Recalling (\ref{alpha}), one has $\alpha\leq{\mu}/\big({2\sum_{i=1}^N(l_i^2\|L_i\|^2){{+}}\gamma{b}}\big)$. It follows that
%where the last inequality is obtained by using $\alpha\leq{\mu}/\big({2\sum_{i=1}^N(l_i^2\|L_i\|^2){{+}}\gamma{b}}\big)$ from (\ref{alpha}).
\begin{equation}\label{eq.Vk{+}1{-}Vk.1.game2}
	\begin{aligned}
		&V({k{+}1}){{-}}V(k)\\
		\leq&{{-}}\frac{\gamma}{4}\|\bm{e_\xi}(k)\|^2{-}\mu\alpha\left\|\hat{L}^2\breve{\mathcal{P}}(\bm{\xi}({k}))\right\|^2.
	\end{aligned}
\end{equation}

Furthermore, the following fact can be verified
\begin{equation}\label{eq.hat L breve P}
\begin{aligned}
&{-}\left\|\hat{L}^2\breve{\mathcal{P}}(\bm{\xi}({k}))\right\|^2\\
%=&{-}\sum_{i=1}^N\left\|L_i^2\breve{\mathcal{P}}_i(\bm{\xi}_i({k}))\right\|^2\\
=&{-}\sum_{i=1}^N\left\|L_i^2\left(\breve{\mathcal{P}}_i(\bm{\xi}_i({k})){-}\frac{\partial f_i}{\partial \bm{x}_i}(\bm{x}(k)){+}\frac{\partial f_i}{\partial \bm{x}_i}(\bm{x}(k))\right)\right\|^2\\
%\leq&\sum_{i=1}^N\bigg({-}\left\|L_i^2\left(\breve{\mathcal{P}}_i(\bm{\xi}_i({k})){-}\frac{\partial f_i}{\partial \bm{x}_i}(\bm{x}(k))\right)\right\|^2\\
%&{-}\left\|L_i^2\frac{\partial f_i}{\partial \bm{x}_i}(\bm{x}(k))\right\|^2\\
%&{+}2\left\|L_i^2\left(\breve{\mathcal{P}}_i(\bm{\xi}_i({k})){-}\frac{\partial f_i}{\partial \bm{x}_i}(\bm{x}(k))\right)\right\|\left\|L_i^2\frac{\partial f_i}{\partial \bm{x}_i}(\bm{x}(k))\right\|\bigg)\\
\leq&\sum_{i=1}^N\bigg(\left\|L_i^2\left(\breve{\mathcal{P}}_i(\bm{\xi}_i({k})){-}\frac{\partial f_i}{\partial \bm{x}_i}(\bm{x}(k))\right)\right\|^2\\
&{-}\frac{1}{2}\left\|L_i^2\frac{\partial f_i}{\partial \bm{x}_i}(\bm{x}(k))\right\|^2\bigg)\\
\leq&\sum_{i=1}^N\bigg(\left\|{L_i}\right\|^4l_i^2\left\|\bm{\xi}_i({k}){-}\bm{1}_{n_i}{\otimes}\bm{x}(k)\right\|^2\\
&{-}\frac{\lambda_2(L_i^2)}{2}\left\|L_i\frac{\partial f_i}{\partial \bm{x}_i}(\bm{x}(k))\right\|^2\bigg)\\
\leq&\max_{i\in\mathcal{I}}\{\|{L_i}\|^4l_i^2\}\|\bm{e_\xi}({k})\|^2{-}\frac{1}{2}\min_{i\in\mathcal{I}}\{\lambda_2(L_i^2)\}V_x(k),\\
%\leq&\max_{i\in\mathcal{I}}\{\|{L_i}\|^4l_i^2\}\|\bm{e_\xi}({k})\|^2\\
%&{-}\frac{\min_{i\in\mathcal{I}}\{\lambda_2(L_i^2)\}}{2}\sum_{i=1}^N\left\|L_i\frac{\partial f_i}{\partial \bm{x}_i}(\bm{x}(k))\right\|^2\\
\end{aligned}
\end{equation}
where the first inequality is derived by using
\begin{equation*}
\begin{aligned}
&-2\left(L_i^2\left(\breve{\mathcal{P}}_i(\bm{\xi}_i({k})){-}\frac{\partial f_i}{\partial \bm{x}_i}(\bm{x}(k))\right)\right)^T\left(L_i^2\frac{\partial f_i}{\partial \bm{x}_i}(\bm{x}(k))\right)\\
\leq &2\left\|L_i^2\left(\breve{\mathcal{P}}_i(\bm{\xi}_i({k})){-}\frac{\partial f_i}{\partial \bm{x}_i}(\bm{x}(k))\right)\right\|^2+\frac{1}{2}\left\|L_i^2\frac{\partial f_i}{\partial \bm{x}_i}(\bm{x}(k))\right\|^2,
\end{aligned}
\end{equation*}
the second inequality is obtained from
% ${-}\bm{x}^TL^2\bm{x}\leq{-}\lambda_2(L)\bm{x}^TL\bm{x}$ and
(\ref{eq.breve.Pi-partial.fi}) in the appendix, and $\lambda_2(L_i^2)$ denotes the smallest non-zero eigenvalue of the matrix $L_i^2$.
Substituting (\ref{eq.hat L breve P}) back into (\ref{eq.Vk{+}1{-}Vk.1.game2}) yields
\begin{equation*}
	\begin{aligned}
		&V({k{+}1}){{-}}V(k)\\
		%\leq&{{-}}\frac{\gamma}{4}\|\bm{e_\xi}(k)\|^2{-}l\alpha\left\|\hat{L}^2\breve{\mathcal{P}}(\bm{\xi}({k}))\right\|^2\\
		%\leq&{{-}}\frac{\gamma}{4}\|\bm{e_\xi}(k)\|^2{+}l\alpha\bigg(\max_{i\in\mathcal{I}}\{\|{L_i}\|^4l_i^2\}\|\bm{e_\xi}({k})\|^2\\
		%&{-}\frac{\min_{i\in\mathcal{I}}\{\lambda_2(L_i^2)\}}{2}\sum_{i=1}^N\left\|L_i\frac{\partial f_i}{\partial \bm{x}_i}(\bm{x}(k))\right\|^2\bigg)\\
		\leq&{{-}}\big(\frac{\gamma}{4}{-}\mu\alpha\max_{i\in\mathcal{I}}\{\|{L_i}\|^4l_i^2\} \big)\|\bm{e_\xi}(k)\|^2\\
		&{-}\frac{\mu\alpha}{2}{\min\limits_{i\in\mathcal{I}}\{\lambda_2(L_i^2)\}}V_x(k).
	\end{aligned}
\end{equation*}
Note from (\ref{alpha}) that  $\alpha\leq{\gamma}/({8\mu\max_{i\in\mathcal{I}}\{\|{L_i}\|^4l_i^2\}})$. It follows that
\begin{equation*}
\begin{aligned}
&V({k{+}1}){{-}}V(k)\\
\leq&{{-}}\frac{\gamma}{8}\|\bm{e_\xi}(k)\|^2{-}\frac{\mu\alpha}{2}{\min\limits_{i\in\mathcal{I}}\{\lambda_2(L_i^2)\}}V_x(k)\\
\leq&{-}\varepsilon_1 V(k),
\end{aligned}
\end{equation*}
where
$$\varepsilon_1=\min\{\frac{1}{8 \|W_\mathcal{M}\|},\frac{\mu\alpha}{2}{\min\limits_{i\in\mathcal{I}}\{\lambda_2(L_i^2)\}}\}.
$$ The above inequality indicates that $V(k)$ will converge to zero with a linear rate $O((1{-}\varepsilon_1)^k)$. Then, one can conclude from Lemma \ref{lemma.NE} that $\bm{x}$ will converge linearly to $\bm{x}^*$.
% has $\lim\limits_{k\rightarrow\infty}L_i\frac{\partial f_i}{\partial \bm{x}_i}(\bm{x}(k))=0,\forall i\in\mathcal{I}$, which implies that $\lim\limits_{k\rightarrow\infty}\bm{x}(k)=\bm{x}^*$.
$\hfill\blacksquare$

\section{Distributed NE computation for the general case}\label{sec.general case}
%\section{Distributed NE computation in the general resource allocation game}
%\section{Distributed NE computation in game (1)}
%To find the NE in game (\ref{eq.problem1}) in a distributed manner, the algorithm design will be more complicated than that in game (\ref{eq.problem2}).

In this section, we consider \textit{the general case} described by model (\ref{eq.problem1}) that the individual agent benefits may be effected by the decisions of all the game participants. In this case, the equation (\ref{relation}) is no longer valid, which implies that the proposed algorithm in the previous section cannot  ensure the collective agent state converge to the NE. Based partly upon the results provided in the last section, we will design a new algorithm for \textit{the general case} and present the convergence analysis.
\subsection{Algorithm Design}
To make the collective agent state converge to  %distributively
the NE in game (\ref{eq.problem1}), a distributed algorithm is designed for each agent $ ij\in\mathcal{V}~\forall k\in\mathbb{N}$ as follows:
\begin{subequations}\label{eq.law.2}
\begin{align}
x_{ij}(k)=&x_{ij}(0){{-}}\sum_{im\in\mathcal{N}_{ij}^{i}} (\eta_{ij}(k){{-}}\eta_{im}(k)),\label{eq.law.x.ij.2}\\
{\eta}_{ij}(k{{+}}1)=&{\eta}_{ij}(k){+}{\beta}\sum_{im\in\mathcal{N}_{ij}^{i}} \left(\psi_{ij}^{ij}(k){-}\psi_{ij}^{im}(k)\right),\label{eq.law.eta.ij.2}\\
\psi^{il}_{ij}(k{{+}}1)=&\sum_{im\in\mathcal{N}_{ij}^{i}}c_{ij}^{im}\psi_{im}^{il}(k){+}\frac{\partial f_{ij}}{\partial x_{il}}\big(\bm{\xi}_{ij}(k{{+}}1)\big)\nonumber\\
&{{-}}\frac{\partial f_{ij}}{\partial x_{il}}\big(\bm{\xi}_{ij}(k)\big),~~\forall il\in\mathcal{V}_i,\label{eq.law.psi.ijil}\\
\xi_{ij}^{pq}(k{{+}}1)=&\bar{w}_{ij}^{pq}\xi_{ij}^ {pq}(k){{+}}\sum\limits_{lm\in\mathcal{N}_{ij}}w_{ij}^{lm}\xi_{lm}^{pq}(k)\nonumber\\
&{+}w_{ij}^{pq}x_{pq}(k),~~~\forall pq\in\mathcal{V},\label{eq.law.xi.ijpq2}
\end{align}
\end{subequations}
%\end{equation}
where $\beta$ is a positive constant to be determined, $x_{ij}(0)=R_{ij}$, $\eta_{ij}(0){=}0$, $\psi_{ij}^{il}(0)=\frac{\partial f_{ij}}{\partial x_{il}}(\bm{\xi}_{ij}(0)),\forall il\in\mathcal{V}_i$, and other variables are defined the same as in previous sections.
In this algorithm, each agent $ij$ should update the variables $x_{ij}, \eta_{ij},\psi_{ij}^{i1},\cdots,\psi_{ij}^{in_i},\xi_{ij}^{11},\cdots,\xi_{ij}^{Nn_N}$.

Define the vectors
\begin{equation*}
\begin{aligned}
&\bm{\psi}_i=[\psi^{i1}_{i1},\psi^{i2}_{i1},{\cdots},\psi^{in_i}_{i1},\psi_{i2}^{i1},\cdots,\psi_{i2}^{in_i},\cdots,\psi_{in_i}^{in_i}]^T\in\mathbb{R}^{n_i^2},\\%,\forall i\in\{1,2,{\cdots},N\}\\
&\bm{\psi}=[\bm{\psi}_1^T,\bm{\psi}_2^T,{\cdots},\bm{\psi}_N^T]^T\in\mathbb{R}^{n_1^2{+}\cdots{+}n_N^2},
%&\bm{\xi}_{ij}=[\xi_{ij}^{11},\xi_{ij}^{12},\cdots,\xi_{ij}^{1n_1},\xi_{ij}^{21},\cdots,\xi_{ij}^{2n_2},\cdots,\xi_{ij}^{Nn_N}]\in\mathbb{R}^{n_{\mathrm{sum}}},\\
%&\bm{\xi}_i=[(\bm{\xi}_{i1})^T,(\bm{\xi}_{i2})^T,\cdots,(\bm{\xi}_{i{n_i}})^T]^T\in\mathbb{R}^{n_in_{\mathrm{sum}}},\\
%&\bm{\xi}=[(\bm{\xi}_{1})^T,(\bm{\xi}_{2})^T,\cdots,(\bm{\xi}_{N})^T]^T\in\mathbb{R}^{n_{\mathrm{sum}}^2},
\end{aligned}
\end{equation*}
and the function $\mathcal{Q}_i:\mathbb{R}^{n_in_\mathrm{sum}}\rightarrowtail\mathbb{R}^{n_i^2}$:
\begin{equation*}
\begin{aligned}
\mathcal{Q}_i(\bm{\xi}_{i}){=}\left[(\frac{\partial f_{i1}}{\partial \bm x_{i}}(\bm{\xi}_{i1}))^T,(\frac{\partial f_{i2}}{\partial \bm x_{i}}(\bm{\xi}_{i2}))^T,{\cdots},(\frac{\partial f_{in_i}}{\partial \bm x_{i}}(\bm{\xi}_{in_i}))^T\right]^T.
%\mathcal{Q}_i(\bm{\xi}_{i})=&[\frac{\partial f_{i1}}{\partial x_{i1}}(\bm{\xi}_{i1}),\frac{\partial f_{i1}}{\partial x_{i2}}(\bm{\xi}_{i1}),{\cdots},\frac{\partial f_{i1}}{\partial x_{i{n_i}}}(\bm{\xi}_{i1}),\frac{\partial f_{i2}}{\partial x_{i1}}(\bm{\xi}_{i2}),\\
%&\frac{\partial f_{i2}}{\partial x_{i2}}(\bm{\xi}_{i2}),
\end{aligned}
\end{equation*}
Then, the proposed algorithm (\ref{eq.law.2}) can be rewritten in the following compact form $\forall i\in\mathcal{I}$:
\begin{subequations}
\begin{align}
\bm{x}_i(t)=&\bm{x}_i(0){{-}}L_i\bm{\eta}_i(k),\label{eq.law.x.i.2}\\
{\bm{\eta}}_i(k{{+}}1)=&{\bm{\eta}}_i(k){+}\beta\breve{L}_i\bm{\psi}_{i}({k}),\label{eq.law.eta.i.2}\\
{\bm{\psi}}_i(k{{+}}1)=&(C_i{{\otimes}} I_{n_i}){\bm{\psi}}_i(k){+}\mathcal{Q}_i(\bm{\xi}_i(k{{+}}1))\nonumber\\
&{{-}}\mathcal{Q}_i(\bm{\xi}_i({k})),\label{eq.law.psi.i}\\
\bm{\xi}(k{{+}}1)
=&(W{\otimes} I_{n_\mathrm{sum}}{+}{\bar{W}})\bm{\xi}(k){{+}}\hat{W}\big(\bm{1}_{n_\mathrm{sum}}{{\otimes}}\bm{x}(k)\big),\label{eq.law.xi.2}
\end{align}
\end{subequations}
where $\breve{L}_i=\mathrm{diag}\{(L_i)_{1},\cdots,(L_i)_{n_i}\}$ with $(L_i)_{j}$ denoting the $j$th row of the matrix $L_i$. %and $R_=[c_{ij}^{im}]\in\mathbb{R}^{n_i\times n_i}$ denotes the matrix with $c_{ij}^{im}$ being the element on the $j$th row and the $m$th column.

\begin{remark}
	%For the general case of the resource allocation game (\ref{eq.problem1}),  the equation (\ref{relation}) is no longer valid.
	%In the proposed algorithm (\ref{eq.law.2})
%	For the general case considered in this section, the update of agent state $x_{ij}$ in (\ref{eq.law.x.ij.2}) and the global-state estimation $\xi_{ij}^{pq}$ in (\ref{eq.law.xi.ijpq2}) are designed in the same form as in the algorithm (\ref{eq.law.1}) for the special case.
	
	The difference between the algorithm (\ref{eq.law.2}) for the general case and the algorithm (\ref{eq.law.1}) for the special case lies in the design of  auxiliary variables $\eta_{ij}$, $\forall ij\in\mathcal{V}$.
	As  discussed in Remark 1, under (\ref{eq.law.x.ij.2}), the primal problem with the resource constraints can be converted to a new problem  regarding the composite function of $\bm{\eta}$ without any constraint, and then solved based on pseudo-gradient descent. In such a design approach, the update of $\eta_{ij}$ for each agent $ij\in\mathcal{V}$ requires the information of  ${\partial f_i}/{\partial x_{ij}}$ and ${\partial f_i}/{\partial x_{im}},\forall im\in\mathcal{N}_{ij}^{i}$.  However, in the general case,  since the equation (\ref{relation}) is no longer valid,  agent $ij\in\mathcal{V}$ cannot access the exact knowledge of ${\partial f_i}/{\partial x_{ij}}$ and ${\partial f_i}/{\partial x_{im}},\forall im\in\mathcal{N}_{ij}^{i}$.
%	Noticing that each agent $ij\in\mathcal{V}$ cannot access the exact knowledge of ${\partial f_i}/{\partial x_{ij}}$ and ${\partial f_i}/{\partial x_{im}},\forall im\in\mathcal{N}_{ij}^{i}$ in the general case as  the equation (\ref{relation}) is no longer valid,
To overcome this issue,
	the auxiliary variables $\psi_{ij}^{il}~\forall il\in\mathcal{V}_i$ governed by (\ref{eq.law.psi.ijil}) are skillfully integrated into the update of $\eta_{ij}$ in (\ref{eq.law.eta.ij.2}), where $\psi_{ij}^{il}$ is computed by agent $ij$ to estimate the value of $({1}/{n_i})\cdot({\partial f_i}/{\partial x_{il}})(\bm{x})$.
	%we design the updating law of $\eta_{ij}$  as in (\ref{eq.law.eta.ij.2}) with extra auxiliary variables $\psi_{ij}^{il}~\forall il\in\mathcal{V}_i$  governed by (\ref{eq.law.psi.ijil}), where $\psi_{ij}^{il}$ is computed by agent $ij$ to estimate the value of $({1}/{n_i})\cdot({\partial f_i}/{\partial x_{il}})(\bm{x})$.
	The design of (\ref{eq.law.psi.ijil}) is inspired by the gradient tracking technique in distributed optimization \cite{nedic2017}.
\end{remark}

\subsection{Analysis on Steady States}\label{sec.convergence}
Define the following variables for notational brevity:
\begin{equation*}
\begin{aligned}
\bar{\bm{\psi}}_i&=\frac{1}{n_i}(\mathbf{1}_{n_i}^T{{\otimes}}I_{n_i}\big){\bm{\psi}}_i\in\mathbb{R}^{n_i},\\
\bar{\mathcal{Q}}_i(\cdot)&=\frac{1}{n_i}(\mathbf{1}_{n_i}^T{{\otimes}}I_{n_i}){{\mathcal{Q}}_i}(\cdot)\in\mathbb{R}^{n_i}.
\end{aligned}
\end{equation*}
Since the initial value of $\psi_{ij}^{il}$ is set as $\psi_{ij}^{il}(0)=\frac{\partial f_{ij}}{\partial x_{il}}(\bm{\xi}_{ij}(0))$, one has
$
\bm{\psi}_i(0)=\mathcal{Q}_i(\bm{\xi}_i(0))$. Note that  $\bm{1}_{n_i}^TC_i=\bm{1}_{n_i}^T$.	Then, one can derive from  (\ref{eq.law.psi.i}) that
\begin{equation}\label{eq.bar.psi.Q.i}
\bar{\bm{\psi}}_i(k)=\bar{\mathcal{Q}}_i(\bm{\xi}_i(k)),\quad\forall k\in\mathbb{N}.
\end{equation}
One can also obtain by definition that
\begin{equation}\label{eq.bar.Q.i.gradient}
\bar{\mathcal{Q}}_i(\mathbf{1}_{n_i}{{\otimes}}\bm{x})=\frac{1}{n_i}\cdot\frac{\partial f_i}{\partial \bm{x}_i}(\bm{x}).
\end{equation}
The above two equations are quite critical for the forthcoming convergence analysis.

Next, we will present a steady-state analysis of the proposed algorithm, which can facilitate the error system construction and the  convergence analysis.
Suppose that the algorithm variables $\bm{x}_i(k)$, $\bm{\psi}_i(k)$ and $\bm{\xi}(k)$ will settle on some points $\bm{x}_i(\infty)$, $\bm{\psi}_i(\infty)$ and $\bm{\xi}(\infty)$ respectively.
Then, from (\ref{eq.law.eta.i.2}), (\ref{eq.law.psi.i}), and (\ref{eq.law.xi.2}), the steady states satisfy
\begin{align}
&\breve{L}_i\bm{\psi}_{i}({\infty})=0,\label{eq.law.x.i.steady}\\
&{\bm{\psi}}_i(\infty)=(C_i{{\otimes}} I_{n_i}){\bm{\psi}}_i(\infty),\label{eq.law.psi.i.steady}\\
&\left(I_{n_{\mathrm{sum}}^2}{-}\mathcal{M}\right)\big(\bm{\xi}(\infty){{-}}\bm{1}_{n_\mathrm{sum}}{{\otimes}}\bm{x}(\infty)\big)=0.\label{eq.law.xi.steady}
\end{align}
One  can obtain  from (\ref{eq.law.psi.i.steady}) that
%\begin{equation*}
$\bm{\psi}_i(\infty)={\bm{1}_{n_i}}{\otimes} \tau_i,
$
%\end{equation*}
where $\tau_i$ is a constant vector to be determined later. Since
\begin{equation*}
(\mathbf{1}_{n_i}^T{{\otimes}}I_{n_i})\bm{\psi}_i(\infty)=(\mathbf{1}_{n_i}^T{{\otimes}}I_{n_i})({\bm{1}_{n_i}}{\otimes} \tau_i)=n_i\tau_i,
\end{equation*}
one has $
\tau_i=\bar{\bm{\psi}}_i(\infty)$, which implies
\begin{equation}\label{eq.psi.i.infty}
\bm{\psi}_i(\infty)={\bm{1}_{n_i}}{\otimes}\bar{\bm{\psi}}_i(\infty).
\end{equation}
Noting that $\left(I_{n_{\mathrm{sum}}^2}{-}\mathcal{M} \right)$ is non{-}singular,  one can get from (\ref{eq.law.xi.steady}) that
\begin{equation}\label{eq.xi.i.infty}
\bm{\xi}(\infty)=\bm{1}_{n_\mathrm{sum}}{{\otimes}}\bm{x}(\infty).
\end{equation}
%meaning that the action estimates will converge tot the actual states of the players.
Combining (\ref{eq.bar.psi.Q.i}), (\ref{eq.bar.Q.i.gradient}), (\ref{eq.psi.i.infty}) and (\ref{eq.xi.i.infty}), one has
\begin{equation*}
\bm{\psi}_i(\infty)={\bm{1}_{n_i}}{\otimes}\left(\frac{1}{n_i}\cdot\frac{\partial f_i}{\partial \bm{x}_i}(\bm{x}(\infty)\right).
\end{equation*}
Substituting the above equation into (\ref{eq.law.x.i.steady}) yields $$L_i\frac{\partial f_i}{\partial \bm{x}_i}(\bm{x}(\infty))=0.$$
Then one has
$
\bm{x}(\infty)=\bm{x}^*
$ from Lemma \ref{lemma.NE}.
%meaning that the steady states equal to the NE.

\subsection{Error System Construction and Convergence Analysis}

Based on the analysis on steady states,  we define the convergence errors
\begin{equation*}
\begin{aligned}
\bm{e _{\psi_i}}(t)&=\bm{\psi}_i(t){{-}}{\bm{1}_{n_i}}{\otimes}\bar{\bm{\psi}}_i(t),%\in\mathbb{R}^{n_i^2},\\
%\bm{e_{\xi_i}}(t)&=\bm{\xi}_i(t){{-}}\bm{1}_{n_i}{{\otimes}}{\bm{x}}(t)\in\mathbb{R}^{n_in_\mathrm{sum}},\\
\end{aligned}
\end{equation*}
and $\bm{e_\psi}=[\bm{e_{\psi_1}}^T,\bm{e_{\psi_2}}^T,{\cdots},\bm{e_{\psi_N}}^T]^T$.
%\textit{i) The iteration of $\mathbf{e_{\psi_i}}$}:\\
%\paragraph{The iteration of $\bm{e_{\psi_i}}$}
One can obtain from the iteration of $\bm{{\psi}}_i$ in (\ref{eq.law.psi.i})  that:
\begin{equation}\label{eq.e_psi.i.k{+}1.k}
\begin{aligned}
&\bm{e_{\psi_i}}(k{{+}}1)\\
%=&\bm{\psi}_i(k{{+}}1){{-}}{\bm{1}_{n_i}}{\otimes}\bar{\bm{\psi}}_i(k{{+}}1)\\
%=&(C_i{{\otimes}} I_{n_i}){\bm{\psi}}_i(k){+}\mathcal{Q}_i(\bm{\xi}_i(k{{+}}1)){{-}}\mathcal{Q}_i(\bm{\xi}_i({k}))\\
%&{{-}}{\bm{1}_{n_i}}{\otimes}\Big(\frac{1}{n_i}(\bm{1}_{n_i}^T{{\otimes}} I_{n_i})\Big((C_i{{\otimes}} I_{n_i}){\bm{\psi}}_i(k)\\
%&{+}\mathcal{Q}_i(\bm{\xi}_i(k{{+}}1)){{-}}\mathcal{Q}_i(\bm{\xi}_i({k}))\Big)\Big)\\
%=&(C_i{{\otimes}} I_{n_i}){\bm{\psi}}_i(k){+}\mathcal{Q}_i(\bm{\xi}_i(k{{+}}1)){{-}}\mathcal{Q}_i(\bm{\xi}_i({k}))\\
%&{{-}}{\bm{1}_{n_i}}{\otimes}\Big(\frac{1}{n_i}(\bm{1}_{n_i}^T{{\otimes}} I_{n_i})\Big({\bm{\psi}}_i(k){+}\mathcal{Q}_i(\bm{\xi}_i(k{{+}}1)){{-}}\mathcal{Q}_i(\bm{\xi}_i({k}))\Big)\Big)\\
=&(C_i{{\otimes}} I_{n_i}){\bm{\psi}}_i(k){{-}}{\bm{1}_{n_i}}{{\otimes}}\bar{\bm{\psi}}_i(k){+}\mathcal{Q}_i(\bm{\xi}_i(k{{+}}1)){{-}}\mathcal{Q}_i(\bm{\xi}_i({k}))\\
&{{-}}{\bm{1}_{n_i}}{\otimes}\Big(\frac{1}{n_i}(\bm{1}_{n_i}^T{{\otimes}} I_{n_i})\big(\mathcal{Q}_i(\bm{\xi}_i(k{{+}}1)){{-}}\mathcal{Q}_i(\bm{\xi}_i({k}))\big)\Big)\\
%=&(C_i{{\otimes}} I_{n_i}){\bm{\psi}}_i(k){{-}}(C_i{{\otimes}} I_{n_i})\left({\bm{1}_{n_i}}{{\otimes}}\bar{\bm{\psi}}_i(k)\right)\\
%&{+}\mathcal{Q}_i(\bm{\xi}_i(k{{+}}1)){{-}}\mathcal{Q}_i(\bm{\xi}_i({k}))\\
%&{{-}}\left({\bm{1}_{n_i}}{\otimes}(\frac{1}{n_i}(\bm{1}_{n_i}^T{{\otimes}} I_{n_i}))\right)\big(\mathcal{Q}_i(\bm{\xi}_i(k{{+}}1)){{-}}\mathcal{Q}_i(\bm{\xi}_i({k}))\big)\\
=&(C_i{{\otimes}} I_{n_i})\bm{e_{\psi_i}}(k)
{+}\mathcal{Q}_i(\bm{\xi}_i(k{{+}}1)){{-}}\mathcal{Q}_i(\bm{\xi}_i({k}))\\
&{{-}}\left(\frac{{\bm{1}_{n_i}}\bm{1}_{n_i}^T}{n_i}{{\otimes}} I_{n_i}\right)\big(\mathcal{Q}_i(\bm{\xi}_i(k{{+}}1)){{-}}\mathcal{Q}_i(\bm{\xi}_i({k}))\big)\\
%=&\left((C_i{{-}}\frac{{\bm{1}_{n_i}}\bm{1}_{n_i}^T}{n_i}){{\otimes}} I_{n_i}\right)\bm{e_{\psi_i}}(k)\\
%&{+}\left((I_{n_i}{{-}}\frac{{\bm{1}_{n_i}}\bm{1}_{n_i}^T}{n_i}){{\otimes}} I_{n_i}\right)\big(\mathcal{Q}_i(\bm{\xi}_i(k{{+}}1)){{-}}\mathcal{Q}_i(\bm{\xi}_i({k}))\big)\\
=&\left(\bar{C}_i{{\otimes}} I_{n_i}\right)\bm{e_{\psi_i}}(k){+}\left(\bar{I}_{i}{{\otimes}} I_{n_i}\right)\big(\mathcal{Q}_i(\bm{\xi}_i(k{{+}}1)){{-}}\mathcal{Q}_i(\bm{\xi}_i({k}))\big),
\end{aligned}
\end{equation}
where
$
\bar{C}_i=C_i{{-}}\frac{{\bm{1}_{n_i}}\bm{1}_{n_i}^T}{n_i}
$,
$
\bar{I}_{i}=I_{n_i}{{-}}\frac{{\bm{1}_{n_i}}\bm{1}_{n_i}^T}{n_i}
$,
and the last equality is derived by using the fact that
$$(\frac{{\bm{1}_{n_i}}\bm{1}_{n_i}^T}{n_i}{{\otimes}} I_{n_i})\bm{e_{\psi_i}}(k)=0.$$
Under Assumption \ref{assp.graph}, one has $\lim_{k\rightarrow\infty}C_i^k=\mathbf{1}_{n_i}\mathbf{1}_{n_i}^T/n_i$, which further implies that $\lim_{k\rightarrow\infty}\bar{C}_i^k=0$. Then, it is obvious that
$\bar{C}_i$ is a Schur matrix. Therefore, there exists a symmetric
positive definite matrix $W_{c_i}$ such that
$\bar{C}_i^T W_{c_i}\bar{C}_i {{-}}W_{c_i} = {{-}}  I_{n_i}$.

From (\ref{eq.law.x.i.2}) and (\ref{eq.law.eta.i.2}), one has
\begin{equation}
\bm{x}_i(k{{+}}1){-}\bm{x}_i(k)={{-}}\beta L_i\breve{L}_i\bm{\psi}_{i}({k}),
\end{equation}
which can be rewritten  in the following collective form
\begin{equation}\label{eq.x.k{+}1.k.2}
\bm{x}(k{{+}}1){-}\bm{x}(k)={{-}}\beta\hat{L}\hat{\breve{L}}\bm{\psi}({k}),
\end{equation}
where $\hat{L}=\mathrm{diag}\{L_1,\cdots,L_N\}$ and  $\hat{\breve{L}}=\mathrm{diag}\{\breve{L}_1,\cdots,\breve{L}_N\}$.
Then, under the algorithm designed in this section, for the estimate error $\bm{e_\xi}$ defined in (\ref{eq.exi.def}), one can derive the following equality $\forall k\in\mathbb{N}$:
\begin{equation}\label{eq.e_xi.k{+}1.k.2}
\begin{aligned}
&\bm{e_\xi}(k{{+}}1)
{=}\mathcal{M}\bm{e_\xi}(k){{+}}\bm{1}_{n_\mathrm{sum}}{{\otimes}}(\beta\hat{L}\hat{\breve{L}}\bm{\psi}({k})),
\end{aligned}
\end{equation}
where $\mathcal{M}$ has been defined in the previous section.

\begin{theorem}\label{theo.2}
	Suppose that Assumptions \ref{assp.fij.lipschitz}{-}\ref{assp.graph} hold. %and graph $\mathcal{G}$ is undirected and connected. Then,  has similar structure with
	Under the proposed DRA algorithm  (\ref{eq.law.2}),
	%(\ref{eq.law.1})
	the collective agent state $\bm{x}$ will converge linearly to the NE of the resource allocation game (\ref{eq.problem1}), if $\beta$ satisfies
	\begin{equation}\label{beta}
	\begin{aligned}
	&	\beta\leq \min\bigg\{\frac{\mu}{2\left(\sum_{i=1}^N({l_i^2\|L_i\|^2}/{n_i}){+}\gamma_\xi b\right)},\\
	&	\frac{\gamma_\psi}{8\mu\max\limits_{i\in\mathcal{I}}\{\|{L_i\breve{L}_i}\|^2\}},\frac{\gamma_\xi}{8\mu\max\limits_{i\in\mathcal{I}}\{{\|{L_i}\|^4(\sum_{j=1}^{n_i}l_{ij}^2)}/{n_i^2}\}}\bigg\},
	\end{aligned}
	\end{equation}	
	where
	\begin{equation*}
	\begin{aligned}
	\gamma_\psi=&4\max_{i\in\mathcal{I}}\{n_i\|\breve{L}_i\|^2\},~b={n_\mathrm{sum}}(2\|\mathcal{M}^T W_\mathcal{M}\|^2{{+}}\|W_\mathcal{M}\|),\\
	\gamma_\xi=&4\Big(\max_{i\in\mathcal{I}}\{\frac{1}{n_i}(\sum_{j=1}^{n_i}l_{ij}^2)\left\|L_i\right\|^2\}{+}2\gamma_\psi\max_{i\in\mathcal{I},ij\in\mathcal{V}_i}\{(2\|\bar{C}_i^TW_{c_i}\bar{I}_i\|^2\\
	&{+}\|\bar{I}_i^T W_{c_i} \bar{I}_i\|)l_{ij}^2\}
	\|I_{n_{\mathrm{sum}}}{{-}}\mathcal{M}\|^2\Big).
	\end{aligned}
	\end{equation*}
	Moreover, the equality constraint in (\ref{eq.problem1}) is always satisfied during the iterations.
\end{theorem}

Before proceeding, we present some useful lemmas.

\begin{lemma}\label{lemma.Vxi.2}
	Under Assumption \ref{assp.graph} and  the proposed algorithm (\ref{eq.law.2}), for the function $V_{\xi}(k)$ defined in (\ref{eq.V.xi}),
	the following inequality holds $\forall k\in\mathbb{N}$:
	\begin{equation*}\label{eq.V.xi.k{+}1.k.2}
	\begin{aligned}
	V_{\xi}({k{+}1}){{-}}V_{\xi}(k)\leq{{-}}\frac{1}{2}\|\bm{e_\xi}(k)\|^2{{+}}\beta^2{b}\|\hat{L}\hat{\breve{L}}\bm{\psi}({k})\|^2.
	\end{aligned}
	\end{equation*}
\end{lemma}
The proof of this lemma is omitted, as it is similar to that of Lemma \ref{lemma.Vxi}.
\begin{lemma}\label{lemma.Vpsi}
	Under Assumption \ref{assp.graph} and  the proposed algorithm (\ref{eq.law.2}), for the function
	\begin{equation*}
	V_\psi(k)=\bm{e_\psi}(k)^TW_{c}\bm{e_{\psi}}(k),
	\end{equation*}
	where $W_{c}=\mathrm{diag}\{W_{c_1}{{\otimes}}I_{n_1},\cdots,W_{c_N}{{\otimes}}I_{n_N}\}$,
	the following inequality holds $\forall k\in\mathbb{N}$:
	\begin{equation*}\label{eq.V_psi.k{+}1.k.1}
	\begin{aligned}
	&V_{\psi}(k{{+}}1){{-}}V_{\psi}(k)\\
	{\leq}&{{-}}\frac{1}{2}\|\bm{e_\psi}(k)\|^2{{+}}2\max_{{i\in\mathcal{I},ij\in\mathcal{V}_i}}\{(2\|\bar{C}_i^TW_{c_i}\bar{I}_i\|^2{{+}}\|\bar{I}_i^T W_{c_i} \bar{I}_i\|)l_{ij}^2\}\\
	&\times\|I_{n_{\mathrm{sum}}}{{-}}\mathcal{M}\|^2\|\bm{e_\xi}(k)\|^2.
	\end{aligned}
	\end{equation*}
%		\begin{equation*}\label{eq.V_psi.k{+}1.k.1}
%	\begin{aligned}
%	&V_{\psi}(k{{+}}1){{-}}V_{\psi}(k)\\
%	{\leq}&{{-}}\frac{1}{2}\|\bm{e_\psi}(k)\|^2{{+}}2\max_{i\in\mathcal{I}}\{(2\|\bar{C}_i^TW_{c_i}\bar{I}_i\|^2{{+}}\|\bar{I}_i^T W_{c_i} \bar{I}_i\|)\max_{ij\in\mathcal{V}_i}\{l_{ij}^2\}\}\\
%&\times
%\|I_{n_{\mathrm{sum}}}{{-}}\mathcal{M}\|^2\|\bm{e_\xi}(k)\|^2.
%	\end{aligned}
%	\end{equation*}
\end{lemma}
The proof of Lemma \ref{lemma.Vpsi} is reported in Appendix \ref{proof.lemma.Vpsi}.

\begin{lemma}\label{lemma.barVx}
	Under Assumption \ref{assp.graph} and  the proposed algorithm (\ref{eq.law.2}), for the function
	\begin{equation*}
	\bar V_x(k)=\sum_{i=1}^N\frac{1}{2n_i}\left\|L_i\frac{\partial f_i}{\partial \bm{x}_i}(\bm{x}(k))\right\|^2,
	\end{equation*}
	the following inequality holds $\forall k\in\mathbb{N}$:
	\begin{equation*}
	\begin{aligned}
	&\bar V_x(k{{+}}1){-}\bar V_x(k)\\
	\leq&{-}(\frac{\mu}{\beta}{-}\sum_{i=1}^N\frac{l_i^2\|L_i\|^2}{n_i})\beta^2\left\|\hat{L}\hat{\breve{L}}\bm{\psi}({k})\right\|^2\\
	&{+}\max_{i\in\mathcal{I}}\big\{\frac{1}{n_i}(\sum_{j=1}^{n_i}l_{ij}^2)\left\|L_i\right\|^2\big\}\left\|\bm{e_\xi}(k)\right\|^2{+}\frac{\gamma_\psi}{4}\left\|\bm{e_\psi}(k)\right\|^2.
	\end{aligned}
	\end{equation*}
\end{lemma}
The proof of Lemma \ref{lemma.barVx} is given in Appendix \ref{proof.lemma.barVx}.

Now we are in the position to demonstrate Theorem \ref{theo.2}.\\
\textbf{Proof of Theorem \ref{theo.2}}:

Consider the following Lyapunov function
\begin{equation*}
\tilde V(k)=\bar V_x(k){+}\gamma_\psi V_\psi(k){+}\gamma_\xi V_\xi(k),
\end{equation*}
where $V_\xi(k)$, $V_\psi(k)$, and  $\bar V_x(k)$ have been defined in (\ref{eq.V.xi}), Lemmas \ref{lemma.Vpsi} and \ref{lemma.barVx} respectively, and $\gamma_\psi,\gamma_\xi$ have been given in Theorem \ref{theo.2}. One can obtain from (\ref{beta}) that $\beta\leq{\mu}/\Big({2\big(\sum_{i=1}^N({l_i^2\|L_i\|^2}/{n_i}){+}\gamma_\xi b\big)}\Big)$. Then,
combining  Lemmas \ref{lemma.Vxi.2}, \ref{lemma.Vpsi} and \ref{lemma.barVx}  yields
% Noticing from (\ref{beta}) that $\beta\leq{\mu}/\Big({2\big(\sum_{i=1}^N({l_i^2\|L_i\|^2}/{n_i}){+}\gamma_\xi b\big)}\Big)$, one can obtain from Lemmas \ref{lemma.Vxi.2}, \ref{lemma.Vpsi} and \ref{lemma.barVx} that
\begin{equation}\label{eq.V.k+1.k.2}
\begin{aligned}
&\tilde V({k{+}1}){{-}}\tilde V(k)\\
\leq&{-}\frac{\mu\beta}{2}\left\|\hat{L}\hat{\breve{L}}\bm{\psi}({k})\right\|^2{{-}}\frac{\gamma_\psi}{4}\|\bm{e_\psi}(k)\|^2{{-}}\frac{\gamma_\xi}{4}\|\bm{e_\xi}(k)\|^2.
\end{aligned}
\end{equation}
%where  $\beta\leq{\mu}/\Big({2\big(\sum_{i=1}^N({l_i^2\|L_i\|^2}/{n_i}){+}\gamma_\xi b\big)}\Big)$ from (\ref{beta}) have been used to obtain the inequality.
%
Noting by definitions of $\breve{L}_i$ and $\bm{e_{\psi_{i}}}$, one has
\begin{equation}\label{eq.breve.Li.e.psi.i}
\begin{aligned}
\breve{L}_i\bm{e_{\psi_i}}
=\breve{L}_i(\bm{\psi}_i{-}\bm{1}_{n_i}{\otimes}\bar{\bm{\psi}}_i)
%=&\breve{L}_i\bm{\psi}_i{-}\breve{L}_i(\bm{1}_{n_i}{\otimes}\bar{\bm{\psi}}_i)\\
=\breve{L}_i\bm{\psi}_i{-}{L}_i\bar{\bm{\psi}}_i.
\end{aligned}
\end{equation}
From (\ref{eq.bar.psi.Q.i}) and (\ref{eq.breve.Li.e.psi.i}), one can derive that
\begin{equation}\label{eq.-norm}
\begin{aligned}
&{-}\left\|\hat{L}\hat{\breve{L}}\bm{\psi}({k})\right\|^2={-}\sum_{i=1}^N\left\|L_i\breve{L}_i\bm{\psi}_{i}({k})\right\|^2\\
%=&{-}\sum_{i=1}^N\bigg\|{L_i}\bigg(\breve{L}_i\bm{\psi}_{i}({k}){-}{L}_i\bar{\bm{\psi}}_i(k){+}{L}_i\bar{\mathcal{Q}}_i(\bm{\xi}_i)\\
%&{-}\frac{1}{n_i}L_i\frac{\partial f_i}{\partial \bm{x}_i}(\bm{x}(k)){+}\frac{1}{n_i}L_i\frac{\partial f_i}{\partial \bm{x}_i}(\bm{x}(k))\bigg)\bigg\|^2\\
=&{-}\sum_{i=1}^N\bigg\|{L_i}\bigg(\breve{L}_i\bm{e_{\psi_i}}(k){+}{L}_i\bar{\mathcal{Q}}_i(\bm{\xi}_i){-}\frac{1}{n_i}L_i\frac{\partial f_i}{\partial \bm{x}_i}(\bm{x}(k))\bigg)\\
&{+}\frac{1}{n_i}L_i^2\frac{\partial f_i}{\partial \bm{x}_i}(\bm{x}(k))\bigg\|^2.
\end{aligned}
\end{equation}
Then, by taking similar steps as in (\ref{eq.hat L breve P}), one can get
\begin{equation}\label{eq.-norm}
\begin{aligned}
&{-}\left\|\hat{L}\hat{\breve{L}}\bm{\psi}({k})\right\|^2\\
%\leq&\sum_{i=1}^N\bigg({-}\bigg\|{L_i}\big(\breve{L}_i\bm{e_{\psi_i}}(k){+}{L}_i\bar{\mathcal{Q}}_i(\bm{\xi}_i){-}\frac{1}{n_i}L_i\frac{\partial f_i}{\partial \bm{x}_i}(\bm{x}(k))\big)\bigg\|^2\\
%&{+}\frac{2}{n_i}\bigg\|{L_i}\big(\breve{L}_i\bm{e_{\psi_i}}(k){+}{L}_i\bar{\mathcal{Q}}_i(\bm{\xi}_i){-}\frac{1}{n_i}L_i\frac{\partial f_i}{\partial \bm{x}_i}(\bm{x}(k))\big)\bigg\|\times\\
%&\left\|L_i^2\frac{\partial f_i}{\partial \bm{x}_i}(\bm{x}(k))\right\|{-}\frac{1}{n_i^2}\left\|L_i^2\frac{\partial f_i}{\partial \bm{x}_i}(\bm{x}(k)))\right\|^2\bigg)\\
\leq&\sum_{i=1}^N\bigg(\bigg\|{L_i}\big(\breve{L}_i\bm{e_{\psi_i}}(k){+}{L}_i\bar{\mathcal{Q}}_i(\bm{\xi}_i){-}\frac{1}{n_i}L_i\frac{\partial f_i}{\partial \bm{x}_i}(\bm{x}(k))\big)\bigg\|^2\\
&{-}\frac{1}{2n_i^2}\left\|L_i^2\frac{\partial f_i}{\partial \bm{x}_i}(\bm{x}(k)))\right\|^2\bigg)\\
\leq&\sum_{i=1}^N\bigg(2\|{L_i}\breve{L}_i\bm{e_{\psi_i}}(k)\|^2{+}2\bigg\|{L}_i^2\bar{\mathcal{Q}}_i(\bm{\xi}_i)\\
&{-}\frac{1}{n_i}L_i^2\frac{\partial f_i}{\partial \bm{x}_i}(\bm{x}(k))\bigg\|^2{-}\frac{\lambda_2(L_i^2)}{2n_i^2}\left\|L_i\frac{\partial f_i}{\partial \bm{x}_i}(\bm{x}(k)))\right\|^2\bigg)\\
\leq&2\max_{i\in\mathcal{I}}\{\|{L_i\breve{L}_i}\|^2\}\|\bm{e_\psi}({k})\|^2{-}\min_{i\in\mathcal{I}}\{\frac{\lambda_2(L_i^2)}{n_i}\}\bar V_x\\
&{+}2\max_{i\in\mathcal{I}}\{\frac{\|{L_i}\|^4\sum_{j=1}^{n_i}l_{ij}^2}{n_i^2}\}\|\bm{e_\xi}({k})\|^2,
\end{aligned}
\end{equation}
where  %${-}\bm{x}^TL^2\bm{x}\leq{-}\lambda_2(L)\bm{x}^TL\bm{x}$ is used in the second last step, and
(\ref{eq.ni.barQi{-}partial.f.i}) in the appendix  is used in the last step.
Substituting (\ref{eq.-norm}) back into (\ref{eq.V.k+1.k.2}) yields
\begin{equation*}
\begin{aligned}
&\tilde V({k{+}1}){{-}}\tilde V(k)\\
\leq&\mu\beta\max_{i\in\mathcal{I}}\{\|{L_i\breve{L}_i}\|^2\}\|\bm{e_\psi}({k})\|^2-\frac{\mu\beta}{2}\min_{i\in\mathcal{I}}\{\frac{\lambda_2(L_i^2)}{n_i}\}\bar V_x\\
&{+}\mu\beta\max_{i\in\mathcal{I}}\{\frac{\|{L_i}\|^4(\sum_{j=1}^{n_i}l_{ij}^2)}{n_i^2}\}\|\bm{e_\xi}({k})\|^2\\
&{{-}}\frac{\gamma_\psi}{4}\|\bm{e_\psi}(k)\|^2{{-}}\frac{\gamma_\xi}{4}\|\bm{e_\xi}(k)\|^2\\
\leq
&{-}\frac{\mu\beta}{2}\min_{i\in\mathcal{I}}\{\frac{\lambda_2(L_i^2)}{n_i}\}\bar V_x{{-}}\frac{\gamma_\psi}{8}\|\bm{e_\psi}(k)\|^2{{-}}\frac{\gamma_\xi}{8}\|\bm{e_\xi}(k)\|^2\\
\leq&- \varepsilon_2 \tilde V(k),
\end{aligned}
\end{equation*}
where
\begin{equation*}
\varepsilon_2=\min\left\{\frac{\mu\beta}{2}\min_{i\in\mathcal{I}}\{\frac{\lambda_2(L_i^2)}{n_i}\},\frac{1}{8\|W_{c}\|},\frac{1}{8\|W_{\mathcal{M}}\|}\right\},
\end{equation*}
and the second last inequality can be obtained since $\beta$ satisfies (\ref{beta}).
%Therefore, $V(k{{+}}1){\leq}(1{{-}}\varepsilon)V(k)$.
%where $\varepsilon{>}0$. Since $V$ is the sum of three quadratic form,  which will always be non{-}negative, we do not need to worry about that $\varepsilon$ may be larger than $1$, which leads to that  $V(k{{+}}1){\leq}0$. This will never happen, which implies that $\varepsilon{<}1$. Then, one can conclude that
%It is not difficult to obtain that $\varepsilon\in(0,1)$,
Then, by taking similar steps as in the proof of Theorem 1, one can derive that $\bm{x}$  will converge to $\bm{x}^*$ with a linear rate \textbf{$O((1{-}\varepsilon_2)^k)$} .	$\hfill\blacksquare$

\begin{remark}
%Distinct from existing results on DRA over MANs where all the agents are cooperative, the proposed DRA algorithms (\ref{eq.law.1}) and (\ref{eq.law.2}) can deal with the problem raised by the explict relevance of some agents' decisions to other agents¡¯ individual benefits and the conflicts of interest among the agents.  Moreover, under proposed algorithms, the intra-coalition coupled equality constraints can  be satisfied at each iteration. Such a feature is a favorable in online solving practical problems such as economic dispatch of smart grid, where balancing the power supply and demand while seeking the optimal solution is highly desired.
Distinct from existing results on DRA over MANs with cooperative agents, the proposed algorithms (\ref{eq.law.1}) and (\ref{eq.law.2}) can deal with the DRA problem with conflicts of interest among the agents as well as the influence of some agents' decisions on other agents' individual benefits. Moreover, under proposed algorithms, the intra-coalition coupled equality constraints can  be satisfied at each iteration. Such a feature is favorable in online solving some practical problems such as economic dispatch of smart grid with multiple generating units, where balancing the power supply and demand while seeking the optimal solution is highly desired.
\end{remark}

\section{Numerical simulations}\label{sec.simulation}
%%%%%%%%%%%%%%%%%%%%%%%%%%%%%%%%%%%%%%%%%%%%%%%%%%%%%%%%%%%%%%%%%%%%%%%%%%%%%%%%%%%%%%%%%%
\begin{figure}
	\centering
	\begin{tikzpicture}[-,>=stealth',scale=0.85,
	thick,
	node1/.style={circle,fill=blue!30,align=center,font=\small,scale=0.7}, node2/.style={circle,fill=green!40,align=center,minimum height=5pt,minimum  width=5pt,font=\small,scale=0.7},
	node3/.style={circle,fill=red!30,align=center,minimum height=5pt,minimum  width=5pt,font=\small,scale=0.7}]	
	
	\node[node2] (21) at(0,1) {21};
	\node[node2] (22) at(-1,0) {22};
	\node[node2] (25) at(1,0) {25};
	\node[node2] (23) at(-0.6,-1) {23};
	\node[node2] (24) at(0.6,-1) {24};

	\node[node1] (11) at(-3,1) {11};
	\node[node1] (13) at(-3,-1) {13};
	\node[node1] (12) at(-3.7,0) {12};
	\node[node1] (14) at(-2.3,0) {14};

	\node[node3] (31) at(3,1) {31};
	\node[node3] (36) at(3.9,1) {36};
	\node[node3] (33) at(3,-1) {33};
	\node[node3] (34) at(3.9,-1) {34};
	\node[node3] (32) at(2.3,0) {32};
	\node[node3] (35) at(4.6,0) {35};
	
	\draw[-,blue!60] (11) -- (12);
	\draw[-,blue!60] (11) -- (14);
	\draw[-,blue!60] (13) -- (12);
	
	\draw[-,green] (21) -- (22);
	\draw[-,green] (22) -- (23);
	\draw[-,green] (25) -- (24);
	\draw[-,green] (25) -- (21);
	
	\draw[-,red!60] (32) -- (33);
	\draw[-,red!60] (33) -- (36);
	\draw[-,red!60] (33) -- (31);
	\draw[-,red!60] (35) -- (36);
	\draw[-,red!60] (34) -- (35);

	\draw[-,black!60] (23) -- (13);
	\draw[-,black!60] (25) -- (32);
	
	\end{tikzpicture}
	\caption{The underlying network topology among the game participants in cases 1 and 2.}\label{fig.topology}
\end{figure}
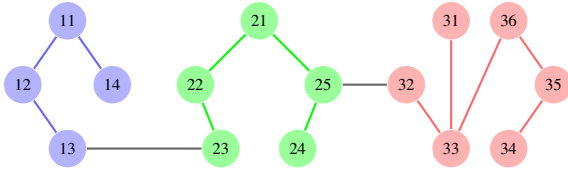

\begin{figure}
	\centering
	\includegraphics[width=3.4in]{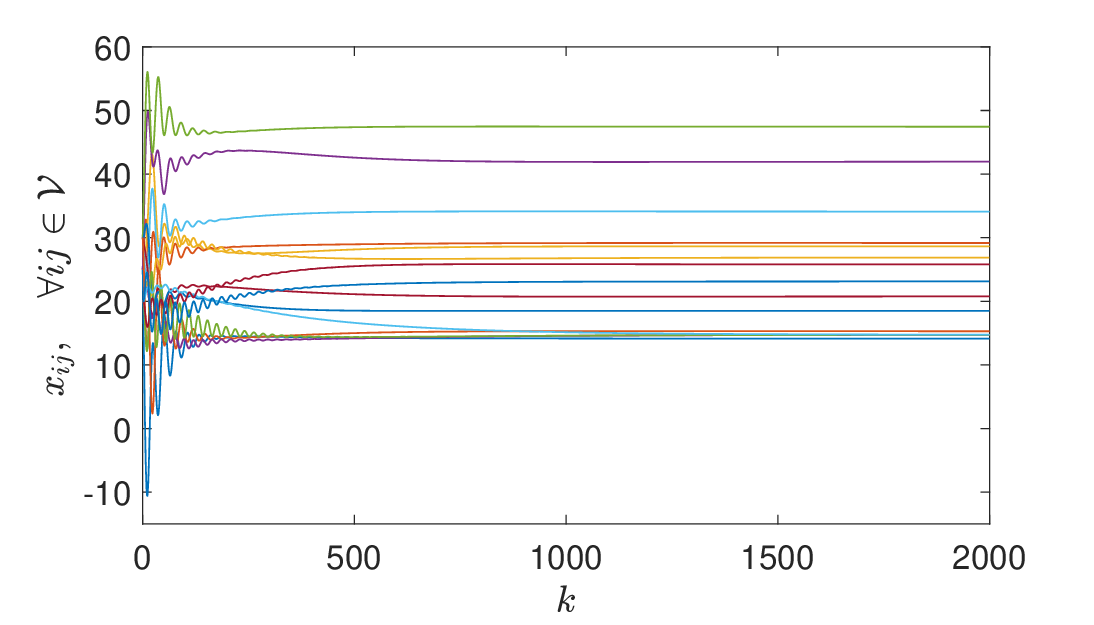}
	\caption{The agent states $x_{ij},\forall ij\in\mathcal{V}$ under the proposed algorithm (\ref{eq.law.1})  in Case 1.}
	\label{fig.xij1}
\end{figure}

\begin{figure}
	\centering
	\includegraphics[width=3.4in]{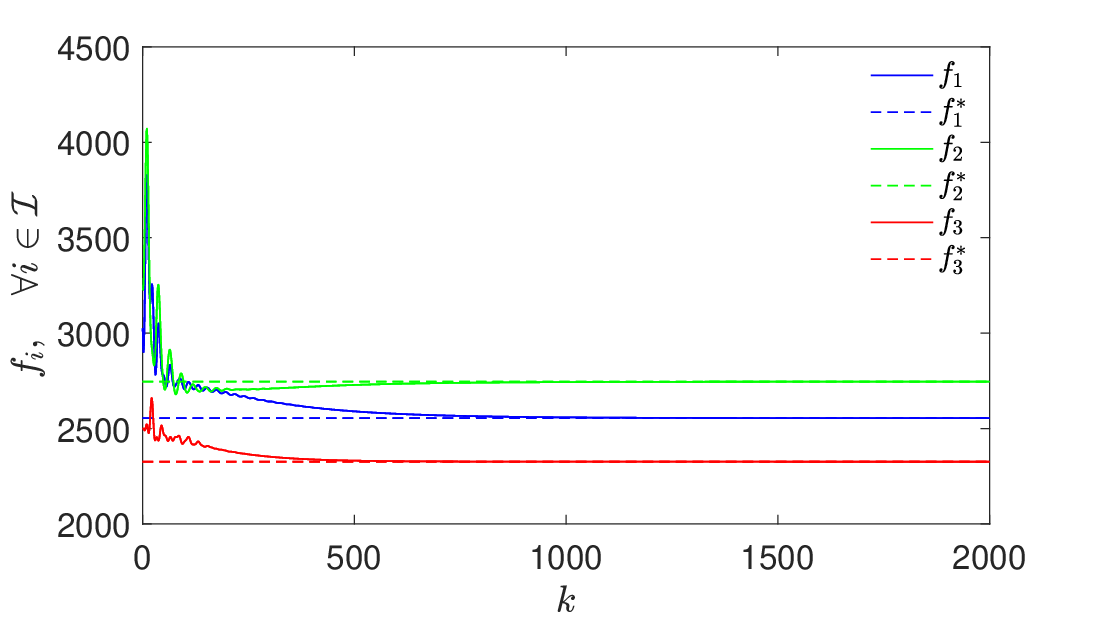}
	\caption{The values of coalition-level objective functions $f_i(\bm{x}),\forall i\in\mathcal{I}$ under the proposed algorithm (\ref{eq.law.1})  in Case 1.}
	\label{fig.fi1}
\end{figure}

\begin{figure}
	\centering
	\includegraphics[width=3.4in]{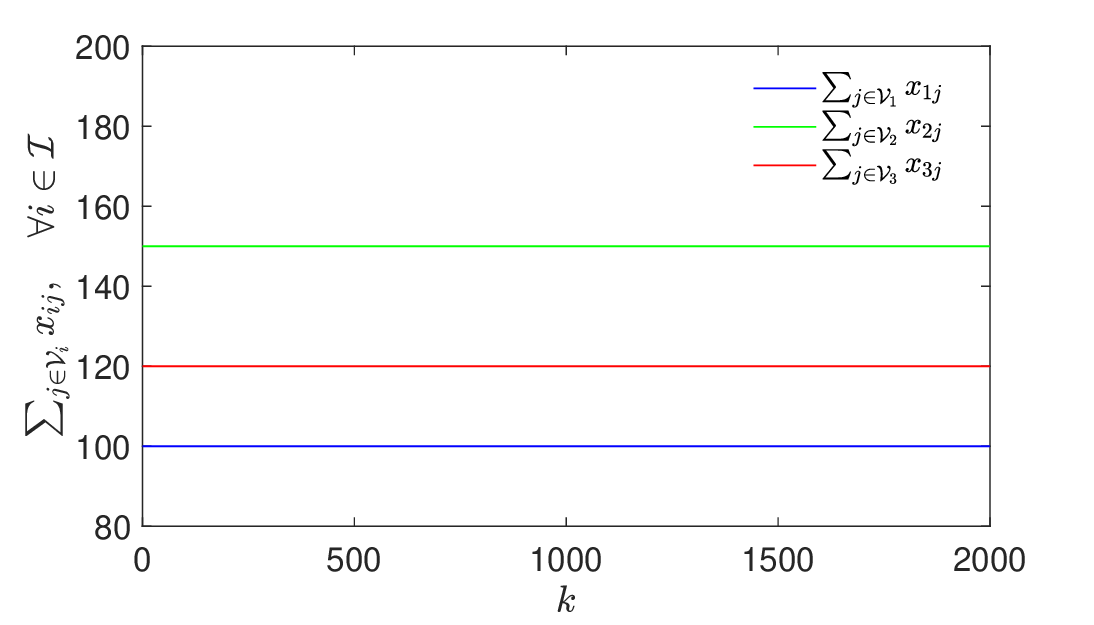}
	\caption{The sum of agent states within each coalition $\sum_{ij\in\mathcal{V}_i}x_{ij},\forall i\in\mathcal{I}$ under the proposed algorithm (\ref{eq.law.1})  in Case 1																																								.}
	\label{fig.sum_xij1}
\end{figure}

\begin{figure}
	\centering
	\includegraphics[width=3.4in]{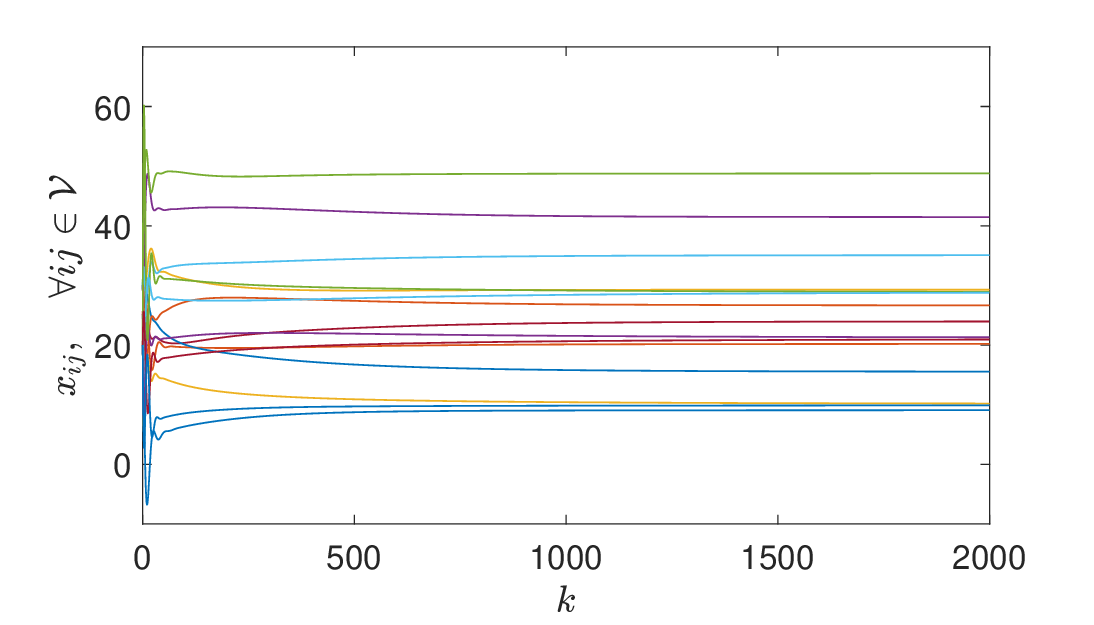}
	\caption{The agent states $x_{ij},\forall ij\in\mathcal{V}$ under the proposed algorithm  (\ref{eq.law.2}) in Case 2.}
	\label{fig.xij2}
\end{figure}

\begin{figure}
	\centering
	\includegraphics[width=3.4in]{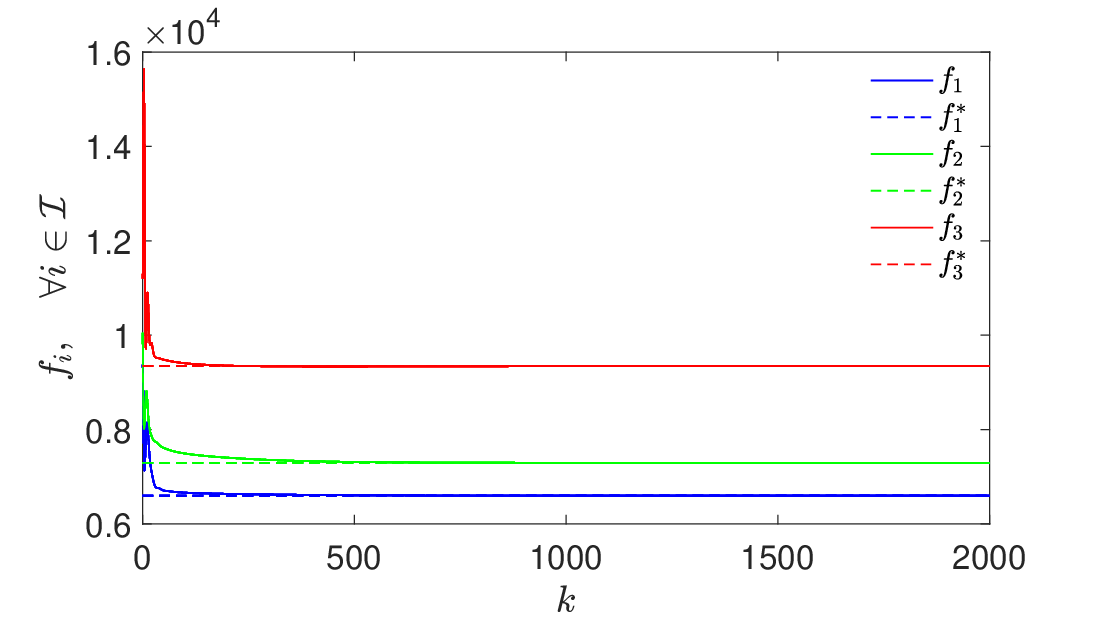}
	\caption{The values of coalition-level objective functions $f_i(\bm{x}),\forall i\in\mathcal{I}$ under the proposed algorithm  (\ref{eq.law.2}) in Case 2.}
	\label{fig.fi2}
\end{figure}

\begin{figure}
	\centering
	\includegraphics[width=3.4in]{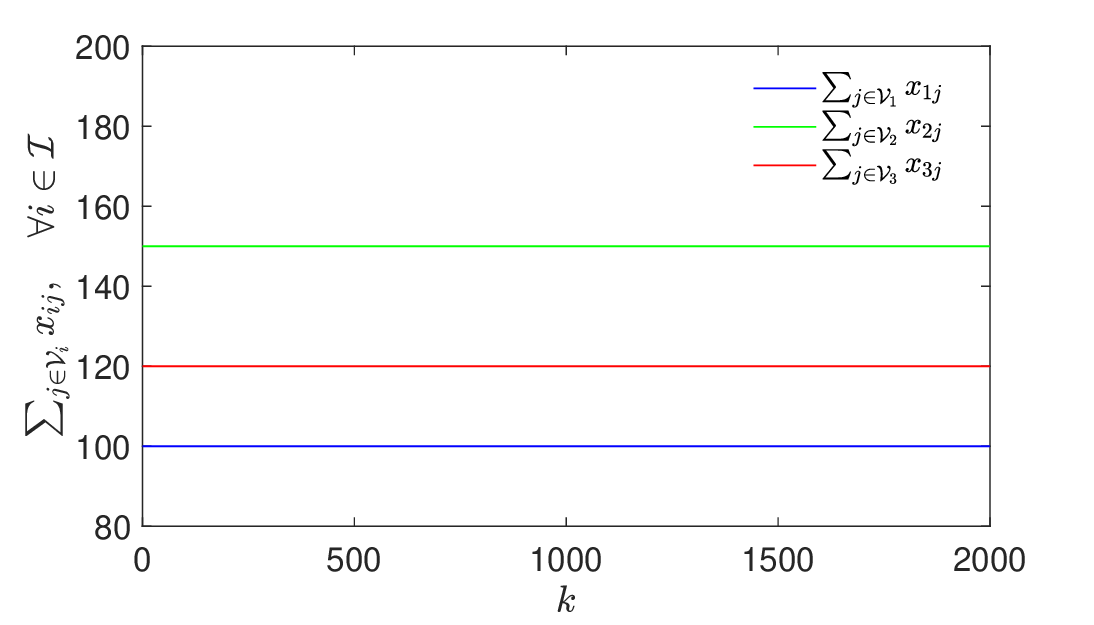}
	\caption{The sum of agent states within each coalition $\sum_{ij\in\mathcal{V}_i}x_{ij},\forall i\in\mathcal{I}$ under the proposed algorithm  (\ref{eq.law.2}) in Case 2																																								.}
	\label{fig.sum_xij2}
\end{figure}

Numerical examples are provided in this section to test the effectiveness of the proposed algorithms. %for both the special and general cases.
Consider three coalitions  that contain  four, five and six agents respectively, i.e., $N=3$, $\mathcal{I}=\{1,2,3\}$, $n_1=4$, $n_2=5$, $n_3=6$. The network topology is shown in Fig. \ref{fig.topology}.
In the following, we consider two cases that can be described by model (\ref{eq.problem1}) and (\ref{eq.problem2}) respectively.

In Case 1, the objective function of each agent $ij\in\mathcal{V}$ is
\begin{equation*}
f_{ij}(x_{ij},\bm{x}_{-i})=(x_{ij}-b_{ij})^2{+}\frac{1}{2}x_{ij}y_{ij},
\end{equation*}
where
$y_{11}{=}x_{31},~y_{12}{=}x_{21}+x_{32},~y_{13}{=}x_{22}+x_{33},~y_{14}{=}x_{23}+x_{34},~ y_{21}{=}x_{12}+x_{32},~y_{22}{=}x_{13}+x_{33},~y_{23}{=}x_{14}+x_{34},~y_{24}{=}x_{35},~y_{25}{=}x_{36},~y_{31}{=}x_{11},~y_{32}{=}x_{12}+x_{21},~y_{33}{=}x_{13}+x_{22},~y_{34}{=}x_{14}+x_{23},~y_{35}{=}x_{24},~y_{36}{=}x_{25}$, and
$b_{11}{=}20$, $b_{12}{=}30$, $b_{13}{=}40$, $b_{14}{=}50$, $b_{21}{=}50$, $b_{22}{=}40$, $b_{23}{=}30$,
$b_{24}{=}20$, $b_{25}{=}30$, $b_{31}{=}b_{32}{=}b_{33}{=}b_{34}{=}b_{35}{=}b_{36}{=}20$.
The quantities of resources in the three coalitions are $R_1=100$, $R_2=150$, and $R_3=120$, respectively.
 Note that in this case, the inputs of the  individual objective function for each agent include only  the states of itself and agents in other coalitions.  One can directly calculate out  the NE
% $\bm{x}^*{=}[  14.1165$,
%$ 15.2945$,
% $28.6278$,
% $41.9612$,
% $47.4431$,
% $34.1098$,
%$ 20.7764$,
%$ 18.5020$,
%$ 29.1687$,
%$ 26.8876$,
%$ 14.7323$,
%$ 14.7323$,
%$ 14.7323$,
%$ 25.7912$,
%$ 23.1245]^T$,
$\bm{x}^*{=}[14.12$,
$15.29$,
$28.63$,
$41.96$,
$47.44$,
$34.11$,
$20.78$,
$18.5$,
$29.17$,
$26.89$,
$14.73$,
$14.73$,
$14.73$,
$25.79$,
$23.12]^T$ and the values of the coalition-level objective functions at the NE  $f_1^*{=}2554,~
f_2^*{=}2746,~ f_3^*{=}2326$. The initial collective state is
$\bm{x}(0){=}[25,25,25,25,30,30,30,30,30,20,20,20,20,20,20]^T$. We employ the proposed algorithm (\ref{eq.law.1}) for the special case with the algorithm parameter set as
$\alpha{=}0.02$. The simulation result is presented in Figs. \ref{fig.xij1}-\ref{fig.sum_xij1}, which show that the agent states converge fast to the NE and the resource constraints are satisfied during the whole process.

In Case 2, the objective function of each agent $ij\in\mathcal{V}$ is
\begin{equation*}
f_{ij}(\bm{x})=5(x_{ij}-d_{ij})^2{+}\frac{1}{2}x_{ij}y_{ij},
\end{equation*}
where
$y_{11}{=}x_{12}+x_{21}+x_{31}+x_{32},~y_{12}{=}x_{11}+x_{21}+x_{31}+x_{32},~y_{13}{=}x_{22}+x_{23}+x_{33}+x_{34},~y_{14}{=}x_{24}+x_{25}+x_{35}+x_{36},~ y_{21}{=}x_{11}+x_{12}+x_{31}+x_{32},~y_{22}{=}x_{13}+x_{23}+x_{33}+x_{34},~y_{23}{=}x_{13}+x_{22}+x_{33}+x_{34},~y_{24}{=}x_{14}+x_{25}+x_{35}+x_{36},~y_{25}{=}x_{14}+x_{24}+x_{35}+x_{36},~y_{31}{=}x_{11}+x_{12}+x_{21}+x_{32},~y_{32}{=}x_{11}+x_{12}+x_{21}+x_{31},~y_{33}{=}x_{13}+x_{22}+x_{23}+x_{34},~y_{34}{=}x_{13}+x_{22}+x_{23}+x_{33},~y_{35}{=}x_{14}+x_{24}+x_{25}+x_{36},~y_{36}{=}x_{14}+x_{24}+x_{25}+x_{35}$, and
$d_{11}{=}20$, $d_{12}{=}30$, $d_{13}{=}40$, $d_{14}{=}50$, $d_{21}{=}50$, $d_{22}{=}40$, $d_{23}{=}30$,
$d_{24}{=}20$, $d_{25}{=}30$, $d_{31}{=}d_{36}{=}20$, $d_{32}{=}d_{35}{=}30$, $d_{33}{=}d_{34}{=}40$.   The quantities of resources in the three coalitions are the same as in Case 1.
By direct calculation, one can obtain the NE
% $\bm{x}^*{=}[ 9.0811$,
%$20.1922$,
%$29.2709$,
%$41.4558$,
%$48.7846$,
%$35.0716$,
%$23.9605$,
%$15.5361$,
%$26.6472$,
%$10.1388$,
%$21.2499$,
%$28.8650$,
%$28.8650$,
%$20.9962$,
%$9.8851]^T$,
$\bm{x}^*{=}[9.08$,
$20.19$,
$29.27$,
$41.46$,
$48.78$,
$35.07$,
$23.96$,
$15.54$,
$26.65$,
$10.14$,
$21.25$,
$28.87$,
$28.87$,
$21.0$,
$9.89]^T$, and the values of the coalition-level objective functions at the NE are $f_1^*{=}6598,~
f_2^*{=} 7295,~ f_3^*=9347$, respectively.
The proposed algorithm (\ref{eq.law.2}) for the general case is employed with the algorithm parameter set as
$\beta{=}0.01$, and the  simulation results are  presented in Figs. \ref{fig.xij2}-\ref{fig.sum_xij2}, showing that the agent states achieve fast convergence to the NE while satisfying the resource constraints.

%%%%%%%%%%%%%%%%%%%%%%%%%%%%%%%%%%%%%%%%%%%%%%%%%%%%%%%%%%%%%%%%%%%%%%%%%%%%%%%%%%%%%%%%%%

\section{Conclusion}\label{sec.conclusion}
%%%%%%%%%%%%%%%%%%%%%%%%%%%%%%%%%%%%%%%%%%%%%%%%%%%%%%%%%%%%%%%%%%%%%%%%%%%%%%%%%%%%%%%%%%
%A class of networked games, termed  consistency{-}constrained multi{-}coalition games, has been investigated in this paper. Each coalition contains multiple agents that aim to minimize the sum of their costs and at the meanwhile reach an agreement on their state values. The property of NE has been analyzed and a distributed discrete{-}time NE computation algorithm is developed, which is effective under general directed networks.

In this paper, the problem of distributed resource allocation over multiple interacting coalitions is investigated by developing game-theoretic approaches.  To characterize the cooperation of individual agents on resource allocation in each coalition as well as the conflicts of interest among different coalitions, a new type of multi-coalition game is formulated.
%Considering the scenarios where the individual benefit of each agent explicitly depends on the states of itself and some agents in other coalitions, and on the states of all the game participants, two distributed DRA algorithms are developed
Inspired by techniques such as variable replacement,
gradient tracking and leader-following consensus, two new kinds of
 DRA algorithms are developed respectively for
the scenarios where the individual benefit of each agent
explicitly depends on the states of itself and some agents
in other coalitions, and on the states of all the game participants. One favourable feature of the designed DRA algorithms is that the resource constraints can be satisfied during the whole allocation process.
%For the special case that the individual benefit of each agent depends explicitly on only the states of itself and agents in other coalitions, a DRA algorithm is designed based on the techniques of variable replacement, gradient descent and leader-following consensus, which can meet the resource constraint during the whole process. Based on this, a DRA algorithm is further developed by integrating the gradient tracking technique for the general case that the individual benefit of each agent may depend explicitly on  all the game participants.
%For the general case that the individual benefit of each agent may depend explicitly on  all the game participants,  a DRA algorithm is further developed by integrating the gradient tracking technique.
Furthermore, linear convergence of the proposed DRA algorithms is successfully  established. In the future, we will consider the cases with directed topologies and time-varying objective functions.
%%%%%%%%%%%%%%%%%%%%%%%%%%%%%%%%%%%%%%%%%%%%%%%%%%
%\section{Some useful lemmas}
\appendix

%\textit{Proof of Lemma \ref{lemma.Vxi}:}\\

\subsection{Proof of Lemma \ref{lemma.Vxi}}\label{proof.lemma.Vxi}
From (\ref{eq.e_xi.k{+}1.k}), one has
\begin{equation*}\label{eq.V.xi.k{+}1.k}
\begin{aligned}
&V_{\xi}({k{+}1}){{-}}V_{\xi}(k)\\
%{=}&\bm{e_\xi}({k{+}1})^TW_{\mathcal{M}}\bm{e_{\xi}}({k{+}1}){{-}}\bm{e_\xi}(k)^TW_{\mathcal{M}}\bm{e_{\xi}}(k)\\
%{=}&\left(\mathcal{M}\bm{e_\xi}(k){{+}}\bm{1}_{n_\mathrm{sum}}{{\otimes}}(\alpha\hat{L}^2\breve{\mathcal{P}}(\bm{\xi}({k})))\right)^T W_{\mathcal{M}}\cdot\\
%& \left(\mathcal{M}\bm{e_\xi}(k){{+}}\bm{1}_{n_\mathrm{sum}}{{\otimes}}(\alpha\hat{L}^2\breve{\mathcal{P}}(\bm{\xi}({k})))\right){{-}}\bm{e_\xi}(k)^TW_{\mathcal{M}}\bm{e_{\xi}}(k)\\
{=}&\bm{e_\xi}^T(t_{k})(\mathcal{M}^T W_\mathcal{M}\mathcal{M} {{-}}W_\mathcal{M} )\bm{e_{\xi}}(t_{k})\\
&{{+}}2\bm{e_\xi}^T(t_{k})\mathcal{M}^T W_\mathcal{M}(\bm{1}_{n_\mathrm{sum}}{{\otimes}}(\alpha\hat{L}^2\breve{\mathcal{P}}(\bm{\xi}({k}))))\\
&{{+}}(\bm{1}_{n_\mathrm{sum}}{{\otimes}}(\alpha\hat{L}^2\breve{\mathcal{P}}(\bm{\xi}({k}))))^TW_\mathcal{M}(\bm{1}_{n_\mathrm{sum}}{{\otimes}}(\alpha\hat{L}^2\breve{\mathcal{P}}(\bm{\xi}({k}))))\\
{\leq}&{{-}}\|\bm{e_\xi}(k)\|^2{{+}}2\sqrt{n_\mathrm{sum}}\|\mathcal{M}^T W_\mathcal{M}\|\|\bm{e_\xi}(k)\|\|\alpha\hat{L}^2\breve{\mathcal{P}}(\bm{\xi}({k}))\|\\
&{{+}}n_{\mathrm{sum}}\|W_\mathcal{M}\|\|\alpha\hat{L}^2\breve{\mathcal{P}}(\bm{\xi}({k}))\|^2\\
%{\leq}&{{-}}\|\bm{e_\xi}(k)\|^2{{+}}2\sqrt{n_\mathrm{sum}}\|\mathcal{M}^T W_\mathcal{M}\|\|\alpha\|\|\bm{e_\xi}(k)\|\|\breve{\bm{\psi}}(k)\|\\
%&{{+}}n_{\mathrm{sum}}\|W_\mathcal{M}\|\|\alpha\|^2\|\breve{\bm{\psi}}(k)\|^2\\
{\leq}&%{{-}}\|\bm{e_\xi}(k)\|^2{{+}}\frac{1}{2}\|\bm{e_\xi}(k)\|^2{{+}}2{n_\mathrm{sum}}\|\mathcal{M}^T W_\mathcal{M}\|^2\|\alpha\hat{L}^2\breve{\mathcal{P}}(\bm{\xi}({k}))\|^2\\
%&{{+}}n_{\mathrm{sum}}\|W_\mathcal{M}\|\|\alpha\hat{L}^2\breve{\mathcal{P}}(\bm{\xi}({k}))\|^2\\
{{-}}\frac{1}{2}\|\bm{e_\xi}(k)\|^2{{+}}\alpha^2b\|\hat{L}^2\breve{\mathcal{P}}(\bm{\xi}({k}))\|^2.
\end{aligned}
\end{equation*}

\subsection{Proof of Lemma \ref{lemma.Vx}}\label{proof.lemma.Vx}
%\textit{Proof of Lemma \ref{lemma.Vx}:}\\
First, one can easily obtain
\begin{equation}\label{eq.Li.partial fi.k{+}1.k}
\begin{aligned}
&\left\|L_i\frac{\partial f_i}{\partial \bm{x}_i}(\bm{x}(k+1))\right\|^2{-}\left\|L_i\frac{\partial f_i}{\partial \bm{x}_i}(\bm{x}(k))\right\|^2\\
%=&\left\|L_i\frac{\partial f_i}{\partial \bm{x}_i}(\bm{x}(k+1)){-}L_i\frac{\partial f_i}{\partial \bm{x}_i}(\bm{x}(k)){+}L_i\frac{\partial f_i}{\partial \bm{x}_i}(\bm{x}(k))\right\|^2\\
%&{-}\left\|L_i\frac{\partial f_i}{\partial \bm{x}_i}(\bm{x}(k))\right\|^2\\
=&\left\|L_i\frac{\partial f_i}{\partial \bm{x}_i}(\bm{x}(k+1)){-}L_i\frac{\partial f_i}{\partial \bm{x}_i}(\bm{x}(k))\right\|^2\\
&{+}2\left(L_i\frac{\partial f_i}{\partial \bm{x}_i}(\bm{x}(k+1)){-}L_i\frac{\partial f_i}{\partial \bm{x}_i}(\bm{x}(k))\right)^T\\
&\times L_i\frac{\partial f_i}{\partial \bm{x}_i}(\bm{x}(k)).
\end{aligned}
\end{equation}
From (\ref{eq.x.i.k{+}1.k}), one can derive that
\begin{equation*}
\begin{aligned}
&2\left(L_i\frac{\partial f_i}{\partial \bm{x}_i}(\bm{x}(k+1)){-}L_i\frac{\partial f_i}{\partial \bm{x}_i}(\bm{x}(k))\right)^TL_i\frac{\partial f_i}{\partial \bm{x}_i}(\bm{x}(k))\\
=&2\left(L_i\frac{\partial f_i}{\partial \bm{x}_i}(\bm{x}(k+1)){-}L_i\frac{\partial f_i}{\partial \bm{x}_i}(\bm{x}(k))\right)^TL_i\bigg(\breve{\mathcal{P}}_i(\bm{\xi}_i)\\
&{+}\frac{\partial f_i}{\partial \bm{x}_i}(\bm{x}(k)){-}\breve{\mathcal{P}}_i(\bm{\xi}_i)\bigg)\\
%=&\left\|L_i\frac{\partial f_i}{\partial \bm{x}_i}(\bm{x}(k+1)){-}L_i\frac{\partial f_i}{\partial \bm{x}_i}(\bm{x}(k))\right\|^2\\
%&{+}2\left(\frac{\partial f_i}{\partial \bm{x}_i}(\bm{x}(k+1)){-}\frac{\partial f_i}{\partial \bm{x}_i}(\bm{x}(k))\right)^TL_i^2\breve{\mathcal{P}}_i(\bm{\xi}_i)\\
%&{+}2\left(L_i\frac{\partial f_i}{\partial \bm{x}_i}(\bm{x}(k+1)){-}L_i\frac{\partial f_i}{\partial \bm{x}_i}(\bm{x}(k))\right)^TL_i\bigg(\frac{\partial f_i}{\partial \bm{x}_i}(\bm{x}(k)){-}\breve{\mathcal{P}}_i(\bm{\xi}_i)\bigg)\\
%=&2\left(\frac{\partial f_i}{\partial \bm{x}_i}(\bm{x}(k+1)){-}\frac{\partial f_i}{\partial \bm{x}_i}(\bm{x}(k))\right)^TL_i^T\frac{1}{\alpha}({\bm{\eta}}_i(k{{+}}1){-}{\bm{\eta}}_i(k))\\
%&{+}2\left(L_i\frac{\partial f_i}{\partial \bm{x}_i}(\bm{x}(k+1)){-}L_i\frac{\partial f_i}{\partial \bm{x}_i}(\bm{x}(k))\right)^T\times\\
%&L_i\bigg(\frac{\partial f_i}{\partial \bm{x}_i}(\bm{x}(k)){-}\breve{\mathcal{P}}_i(\bm{\xi}_i)\bigg)\\
=&{-}\frac{2}{\alpha}\left(\frac{\partial f_i}{\partial \bm{x}_i}(\bm{x}(k+1)){-}\frac{\partial f_i}{\partial \bm{x}_i}(\bm{x}(k))\right)^T\left({\bm{x}}_i(k{{+}}1){{-}}{\bm{x}}_i(k)\right)\\
&{+}2\left(L_i\frac{\partial f_i}{\partial \bm{x}_i}(\bm{x}(k+1)){-}L_i\frac{\partial f_i}{\partial \bm{x}_i}(\bm{x}(k))\right)^T\\
&\times L_i\bigg(\frac{\partial f_i}{\partial \bm{x}_i}(\bm{x}(k)){-}\breve{\mathcal{P}}_i(\bm{\xi}_i)\bigg)\\
\leq&{-}\frac{2}{\alpha}\left(\frac{\partial f_i}{\partial \bm{x}_i}(\bm{x}(k+1)){-}\frac{\partial f_i}{\partial \bm{x}_i}(\bm{x}(k))\right)^T\left({\bm{x}}_i(k{{+}}1){{-}}{\bm{x}}_i(k)\right)\\
&{+}\left\|L_i\frac{\partial f_i}{\partial \bm{x}_i}(\bm{x}(k+1)){-}L_i\frac{\partial f_i}{\partial \bm{x}_i}(\bm{x}(k))\right\|^2\\
&{+}\left\|L_i\bigg(\frac{\partial f_i}{\partial \bm{x}_i}(\bm{x}(k)){-}\breve{\mathcal{P}}_i(\bm{\xi}_i)\bigg)\right\|^2.
\end{aligned}
\end{equation*}
Note that under Assumption \ref{assp.fij.lipschitz}, one has
\begin{equation}\label{eq.breve.Pi-partial.fi}
\begin{aligned}
&\left\|\breve{\mathcal{P}}_i(\bm{\xi}_i(k)){{-}}\frac{\partial f_i}{\partial \bm{x}_i}(\bm{x}(k))\right\|\\
=&\sqrt{\sum\limits_{j=1}^{n_i}\bigg(\frac{\partial f_{i}}{\partial {x}_{ij}}(\bm{\xi}_{ij}(k)){{-}}\frac{\partial f_{i}}{\partial {x}_{ij}}(\bm{x}(k))\bigg)^2}\\
\leq&\sqrt{\sum\limits_{j=1}^{n_i}\big\|\nabla f_{i}(\bm{\xi}_{ij}(k)){{-}}\nabla f_i(\bm{x}(k))\big\|^2}\\
%{\leq}&\sqrt{\sum\limits_{j=1}^{n_i}\left(l_{i}^2\|\bm{\xi}_{ij}(k){{-}}\bm{x}(k)\|^2\right)}\\
%=&\sqrt{l_i^2\sum\limits_{j=1}^{n_i}\|\bm{\xi}_{ij}{{-}}\bm{x}\|^2}\\
\leq &l_i\left\|\bm{\xi}_{i}(k){{-}}\bm{1}_{n_i}{\otimes}\bm{x}(k)\right\|.
\end{aligned}
\end{equation}
Combining the above three formulas yields
\begin{equation*}\label{eq.Vx.k{+}1.k}
\begin{aligned}
&V_x(k{{+}}1){-}V_x(k)\\
%\leq&{-}\frac{2}{\alpha}\sum_{i=1}^N\left(\frac{\partial f_i}{\partial \bm{x}_i}(\bm{x}(k+1)){-}\frac{\partial f_i}{\partial \bm{x}_i}(\bm{x}(k))\right)^T\left({\bm{x}}_i(k{{+}}1){{-}}{\bm{x}}_i(k)\right)\\
%&{+}2\sum_{i=1}^N\left\|L_i\frac{\partial f_i}{\partial \bm{x}_i}(\bm{x}(k+1)){-}L_i\frac{\partial f_i}{\partial \bm{x}_i}(\bm{x}(k))\right\|^2\\
%&{+}\sum_{i=1}^N\left\|L_i\bigg(\frac{\partial f_i}{\partial \bm{x}_i}(\bm{x}(k)){-}\breve{\mathcal{P}}_i(\bm{\xi}_i)\bigg)\right\|^2\\
\leq&{-}\frac{2}{\alpha}(\mathcal{P}(\bm{x}(k{{+}}1)){-}\mathcal{P}(\bm{x}(k)))^T({\bm{x}}(k{{+}}1){-}{\bm{x}}(k))\\
&{+}2\sum_{i=1}^N\left\|L_i\frac{\partial f_i}{\partial \bm{x}_i}(\bm{x}(k+1)){-}L_i\frac{\partial f_i}{\partial \bm{x}_i}(\bm{x}(k))\right\|^2\\
&{+}\sum_{i=1}^N\left\|L_i\bigg(\frac{\partial f_i}{\partial \bm{x}_i}(\bm{x}(k)){-}\breve{\mathcal{P}}_i(\bm{\xi}_i)\bigg)\right\|^2\\
\leq&{-}\frac{2\mu}{\alpha}\left\|{\bm{x}}(k{{+}}1){{-}}{\bm{x}}(k)\right\|^2\\
&{+}2\sum_{i=1}^N\Big(l_i^2\|L_i\|^2\|{\bm{x}}(k{{+}}1){{-}}{\bm{x}}(k)\|^2\Big)\\
&{+}\sum_{i=1}^N\Big(l_i^2\left\|L_i\right\|^2  \left\|\bm{\xi}_{i}(k){{-}}\bm{1}_{n_i}{\otimes}\bm{x}(k)\right\|^2\Big)\\
\leq&%{-}2(\frac{\mu}{\alpha}{-}\sum_{i=1}^Nl_i^2\|L_i\|^2)\left\|{\bm{x}}(k{{+}}1){{-}}{\bm{x}}(k)\right\|^2\\
%&{+}\max_{i\in\mathcal{I}}\{l_i^2\|L_i\|^2\}\left\|\bm{e_\xi}(k)\right\|^2\\
%=&
{-}2\Big({\mu}{-}{\alpha}\sum_{i=1}^N(l_i^2\|L_i\|^2)\Big)\alpha\left\|\hat{L}^2\breve{\mathcal{P}}(\bm{\xi}({k}))\right\|^2\\
&{+}\max_{i\in\mathcal{I}}\{l_i^2\|L_i\|^2\}\left\|\bm{e_\xi}(k)\right\|^2,
\end{aligned}
\end{equation*}
which completes the proof.

%\textit{Proof of Lemma \ref{lemma.Vpsi}:}\\
\subsection{Proof of Lemma \ref{lemma.Vpsi}}\label{proof.lemma.Vpsi}

One can derive from the iteration of $\bm{e_{\psi_i}}$ in (\ref{eq.e_psi.i.k{+}1.k})  that
\begin{equation*}\label{eq.V_psi.k{+}1.k.0}
\begin{aligned}
&V_{\psi}(k{{+}}1){{-}}V_{\psi}(k)\\
%=&\sum_{i=1}^N(\bm{e_{\psi_i}}(k{{+}}1)^T(W_{c_i}{{\otimes}}I_{n_i})\bm{e_{\psi}}(k{{+}}1)\\
%&{{-}}\bm{e_{\psi_i}}^T(k)(W_{c_i}{{\otimes}}I_{n_i})\bm{e_{\psi}}(k))\\
%=&\sum_{i=1}^N\left(\bar{C}_i{{\otimes}} I_{n_i}\right)\bm{e_{\psi_i}}(k){+}\left(\bar{I}_i{{\otimes}} I_{n_i}\right)\big({\mathcal{Q}}_i(\bm{\xi}_i(k{{+}}1)){{-}}{\mathcal{Q}}_i(\bm{\xi}_i({k}))\big)^T W_c\cdot\\
%&\big(\bar{C}\bm{e_{\psi}}(k){+}\bar{I}\left(\mathcal{Q}(\bm{\xi}(k{{+}}1)){{-}}\mathcal{Q}(\bm{\xi}({k}))\right)\big){{-}}\bm{e_\psi}(k)^TW_c\bm{e_{\psi}}(k)\\
=&\sum_{i=1}^N\bigg(\bm{e_{\psi_i}}^T(k)\big((\bar{C}_i^TW_{c_i}\bar{C}_i{-}W_{c_i}){{\otimes}}I_{n_i}\big)\bm{e_{\psi_i}}(k)\\
&{+}2\bm{e_{\psi_i}}^T(k)(\bar{C}_i^TW_{c_i}\bar{I}_i{{\otimes}}I_{n_i})\left(\mathcal{Q}_i(\bm{\xi}_i(k{{+}}1)){{-}}\mathcal{Q}_i(\bm{\xi}_i({k}))\right)\\
&{+}\left(\mathcal{Q}_i(\bm{\xi}_i(k{{+}}1)){{-}}\mathcal{Q}_i(\bm{\xi}_i({k}))\right)^T(\bar{I}_i^T W_{c_i} \bar{I}_i{{\otimes}}I_{n_i})\\
&{\times}\left(\mathcal{Q}_i(\bm{\xi}_i(k{{+}}1)){{-}}\mathcal{Q}_i(\bm{\xi}_i({k}))\right)\bigg)\\
{\leq}&\sum_{i=1}^N\bigg({-}\|\bm{e_{\psi_i}}(k)\|^2{+}\frac{1}{2}\|\bm{e_{\psi_i}}(k)\|^2\\
&{+}2\|\bar{C}_i^TW_{c_i}\bar{I}_i\|^2\|\mathcal{Q}_i(\bm{\xi}_i(k{{+}}1)){{-}}\mathcal{Q}_i(\bm{\xi}_i({k}))\|^2\\
&{+}\|\bar{I}_i^T W_{c_i} \bar{I}_i\|\|\mathcal{Q}_i(\bm{\xi}_i(k{{+}}1)){{-}}\mathcal{Q}_i(\bm{\xi}_i({k}))\|^2\bigg)\\
%=&{{-}}\|\bm{e_\psi}(k)\|^2{+}\frac{1}{2}\|\bm{e_\psi}(k)\|^2\\
%&{+}
%2\|\bar{C}^TW_c\bar{I}\|^2\left\|\mathcal{Q}(\bm{\xi}(k{{+}}1)){{-}}\mathcal{Q}(\bm{\xi}({k}))\right\|^2\\
%&{+}\|\bar{I}^T W_c \bar{I}\|\left\|\mathcal{Q}(\bm{\xi}(k{{+}}1)){{-}}\mathcal{Q}(\bm{\xi}({k}))\right\|^2\\
=&{{-}}\frac{1}{2}\|\bm{e_\psi}(k)\|^2{+}\sum_{i=1}^N(2\|\bar{C}_i^TW_{c_i}\bar{I}_i\|^2{+}\|\bar{I}_i^T W_{c_i} \bar{I}_i\|)\\
&\times\|\mathcal{Q}_i(\bm{\xi}_i(k{{+}}1)){{-}}\mathcal{Q}_i(\bm{\xi}_i({k}))\|^2.
\end{aligned}
\end{equation*}
Under Assumption \ref{assp.fij.lipschitz}, one has $$\left\|\mathcal{Q}_i(\bm{\xi}_i(k{{+}}1)){{-}}\mathcal{Q}_i(\bm{\xi}_i({k}))\right\|^2
{\leq}\max_{ij\in\mathcal{V}_i}\{l_{ij}^2\}\left\|\bm{\xi}_i(k{{+}}1){{-}}\bm{\xi}_i({k})\right\|^2.$$
%\begin{equation*}
%\begin{aligned}
%&\left\|\mathcal{Q}_i(\bm{\xi}_i(k{{+}}1)){{-}}\mathcal{Q}_i(\bm{\xi}_i({k}))\right\|^2\\
%=&\sum_{j=1}^{n_i}\left\|\frac{\partial f_{ij}}{\partial \bm{x}_i}(\bm{\xi_{ij}}(k{{+}}1)){{-}}\frac{\partial f_{ij}}{\partial \bm{x}_i}(\bm{\xi_{ij}}({k}))\right\|^2\\
%{\leq}&{\sum_{j=1}^{n_i}\left(l_{ij}^2\left\|\bm{\xi}_{ij}(k{{+}}1){{-}}\bm{\xi}_{ij}({k})\right\|^2\right)}\\
%{\leq}&\max_{ij\in\mathcal{V}_{i}}\{l_{ij}^2\}\left\|\bm{\xi}_i(k{{+}}1){{-}}\bm{\xi}_i({k})\right\|^2,
%\end{aligned}
%\end{equation*}
Combining the above inequalities yields
\begin{equation*}\label{eq.V_psi.k{+}1.k.1}
\begin{aligned}
&V_{\psi}(k{{+}}1){{-}}V_{\psi}(k)\\
{\leq}&{{-}}\frac{1}{2}\|\bm{e_\psi}(k)\|^2{+}\sum_{i=1}^N(2\|\bar{C}_i^TW_{c_i}\bar{I}_i\|^2{+}\|\bar{I}_i^T W_{c_i} \bar{I}_i\|)\\
&\times\max_{ij\in\mathcal{V}_{i}}\{l_{ij}^2\}\|\bm{\xi}_i(k{{+}}1){{-}}\bm{\xi}_i({k})\|^2\\
{\leq}&{{-}}\frac{1}{2}\|\bm{e_\psi}(k)\|^2{+}\max_{i\in\mathcal{I},ij\in\mathcal{V}_{i}}\{(2\|\bar{C}_i^TW_{c_i}\bar{I}_i\|^2{+}\|\bar{I}_i^T W_{c_i} \bar{I}_i\|)l_{ij}^2\}\\
&\times\|\bm{\xi}(k{{+}}1){{-}}\bm{\xi}({k})\|^2\\
{\leq}&{{-}}\frac{1}{2}\|\bm{e_\psi}(k)\|^2{+}2\max_{i\in\mathcal{I},ij\in\mathcal{V}_{i}}\{(2\|\bar{C}_i^TW_{c_i}\bar{I}_i\|^2{+}\|\bar{I}_i^T W_{c_i} \bar{I}_i\|)l_{ij}^2\}\\
&\times\|I_{n_{\mathrm{sum}}}{{-}}\mathcal{M}\|^2\|\bm{e_\xi}(k)\|^2,
\end{aligned}
\end{equation*}
where the last inequality is obtained by recalling (\ref{eq.exi.def}), (\ref{eq.x.k{+}1.k.2}), and (\ref{eq.e_xi.k{+}1.k.2}).

%\textit{Proof of Lemma \ref{lemma.barVx}:}\\
\subsection{Proof of Lemma \ref{lemma.barVx}}\label{proof.lemma.barVx}

%By definitions of $\breve{L}_i$ and $\bm{e_{\psi_{i}}}$, one has
%\begin{equation}\label{eq.breve.Li.e.psi.i}
%\begin{aligned}
%\breve{L}_i\bm{e_{\psi_i}}
%=\breve{L}_i(\bm{\psi}_i{-}\bm{1}_{n_i}{\otimes}\bar{\bm{\psi}}_i)
%%=&\breve{L}_i\bm{\psi}_i{-}\breve{L}_i(\bm{1}_{n_i}{\otimes}\bar{\bm{\psi}}_i)\\
%=\breve{L}_i\bm{\psi}_i{-}{L}_i\bar{\bm{\psi}}_i.
%\end{aligned}
%\end{equation}
One can derive that
\begin{equation*}
\begin{aligned}
&2\left(L_i\frac{\partial f_i}{\partial \bm{x}_i}(\bm{x}(k+1)){-}L_i\frac{\partial f_i}{\partial \bm{x}_i}(\bm{x}(k))\right)^TL_i\frac{\partial f_i}{\partial \bm{x}_i}(\bm{x}(k))\\
=&2\left(L_i\frac{\partial f_i}{\partial \bm{x}_i}(\bm{x}(k+1)){-}L_i\frac{\partial f_i}{\partial \bm{x}_i}(\bm{x}(k))\right)^T\bigg(n_i\breve{L}_i\bm{\psi}_i(k)\\
&{+}n_iL_i{\bar{{\bm{\psi}}}}_i(k){-}n_i\breve{L}_i\bm{\psi}_i(k)\\
&{+}L_i\frac{\partial f_i}{\partial \bm{x}_i}(\bm{x}(k)){-}n_iL_i\bar{\mathcal{Q}}_i(\bm{\xi}_i({k}))\bigg)\\
%=&2\left(\frac{\partial f_i}{\partial \bm{x}_i}(\bm{x}(k+1)){-}\frac{\partial f_i}{\partial \bm{x}_i}(\bm{x}(k))\right)^TL_i^T\frac{n_i}{\beta}({\bm{\eta}}_i(k{{+}}1){-}{\bm{\eta}}_i(k))\\
%&{+}2\left(L_i\frac{\partial f_i}{\partial \bm{x}_i}(\bm{x}(k+1)){-}L_i\frac{\partial f_i}{\partial \bm{x}_i}(\bm{x}(k))\right)^T\times\\
%&\bigg(L_i\left(\frac{\partial f_i}{\partial \bm{x}_i}(\bm{x}(k)){-}n_i\bar{\mathcal{Q}}_i(\bm{\xi}_i({k}))\right){-}n_i\breve{L}_i\bm{e_{\psi_i}}(k)\bigg)\\
=&{-}\frac{2n_i}{\beta}\left(\frac{\partial f_i}{\partial \bm{x}_i}(\bm{x}(k+1)){-}\frac{\partial f_i}{\partial \bm{x}_i}(\bm{x}(k))\right)^T\left({\bm{x}}_i(k{{+}}1){{-}}{\bm{x}}_i(k)\right)\\
&{+}2\left(L_i\frac{\partial f_i}{\partial \bm{x}_i}(\bm{x}(k+1)){-}L_i\frac{\partial f_i}{\partial \bm{x}_i}(\bm{x}(k))\right)^T\\
&\times\bigg(L_i\left(\frac{\partial f_i}{\partial \bm{x}_i}(\bm{x}(k)){-}n_i\bar{\mathcal{Q}}_i(\bm{\xi}_i({k}))\right){-}n_i\breve{L}_i\bm{e_{\psi_i}}(k)\bigg)\\
%\leq&{-}\frac{2n_i}{\beta}\left(\frac{\partial f_i}{\partial \bm{x}_i}(\bm{x}(k+1)){-}\frac{\partial f_i}{\partial \bm{x}_i}(\bm{x}(k))\right)^T\left({\bm{x}}_i(k{{+}}1){{-}}{\bm{x}}_i(k)\right)\\
%&{+}\left\|L_i\frac{\partial f_i}{\partial \bm{x}_i}(\bm{x}(k+1)){-}L_i\frac{\partial f_i}{\partial \bm{x}_i}(\bm{x}(k))\right\|^2\\
%&{+}\left\|L_i\left(\frac{\partial f_i}{\partial \bm{x}_i}(\bm{x}(k)){-}n_i\bar{\mathcal{Q}}_i(\bm{\xi}_i({k}))\right){-}n_i\breve{L}_i\bm{e_{\psi_i}}(k)\right\|^2\\
\leq&{-}\frac{2n_i}{\beta}\left(\frac{\partial f_i}{\partial \bm{x}_i}(\bm{x}(k+1)){-}\frac{\partial f_i}{\partial \bm{x}_i}(\bm{x}(k))\right)^T\left({\bm{x}}_i(k{{+}}1){{-}}{\bm{x}}_i(k)\right)\\
&{+}\left\|L_i\frac{\partial f_i}{\partial \bm{x}_i}(\bm{x}(k+1)){-}L_i\frac{\partial f_i}{\partial \bm{x}_i}(\bm{x}(k))\right\|^2\\
&{+}2\left\|L_i\left(\frac{\partial f_i}{\partial \bm{x}_i}(\bm{x}(k)){-}n_i\bar{\mathcal{Q}}_i(\bm{\xi}_i({k}))\right)\right\|^2{+}2\left\|n_i\breve{L}_i\bm{e_{\psi_i}}(k)\right\|^2,\\
\end{aligned}
\end{equation*}
where the first equality is obtained from (\ref{eq.bar.psi.Q.i}), the second equality is obtained from (\ref{eq.law.x.i.2}), (\ref{eq.law.eta.i.2}) and (\ref{eq.breve.Li.e.psi.i}).
%$$\bm{x}_i(k{{+}}1){-}\bm{x}_i(k)={{-}}\beta L_i\breve{L}_i\bm{\psi}_{i}({k})$$
Combining the above formula and (\ref{eq.Li.partial fi.k{+}1.k}) yields
\begin{equation*}
\begin{aligned}
&\left\|L_i\frac{\partial f_i}{\partial \bm{x}_i}(\bm{x}(k+1))\right\|^2{-}\left\|L_i\frac{\partial f_i}{\partial \bm{x}_i}(\bm{x}(k))\right\|^2\\
\leq&{{-}}\frac{2n_i}{\beta}\left(\frac{\partial f_i}{\partial \bm{x}_i}(\bm{x}(k+1)){-}\frac{\partial f_i}{\partial \bm{x}_i}(\bm{x}(k))\right)^T\left({\bm{x}}_i(k{{+}}1){{-}}{\bm{x}}_i(k)\right)\\
&{{+}}2\left\|L_i\frac{\partial f_i}{\partial \bm{x}_i}(\bm{x}(k+1)){{-}}L_i\frac{\partial f_i}{\partial \bm{x}_i}(\bm{x}(k))\right\|^2\\
&{{+}}2\left\|L_i\Big(\frac{\partial f_i}{\partial \bm{x}_i}(\bm{x}(k)){{-}}n_i\bar{\mathcal{Q}}_i(\bm{\xi}_i({k}))\Big)\right\|^2{{+}}2\left\|n_i\breve{L}_i\bm{e_{\psi_i}}(k)\right\|^2.\\
\end{aligned}
\end{equation*}
Under Assumption \ref{assp.fij.lipschitz}, one has
\begin{equation}\label{eq.ni.barQi{-}partial.f.i}
\begin{aligned}
&\left\|\frac{\partial f_i}{\partial \bm{x}_i}(\bm{x}(k)){-}n_i\bar{\mathcal{Q}}_i(\bm{\xi}_i({k}))\right\|\\
%{=}\bigg\|\sum\limits_{j=1}^{n_i}\bigg(\frac{\partial f_{ij}(\bm{\xi}_{ij}(k))}{\partial \bm{x}_i}{{-}}\frac{\partial f_{ij}}{\partial \bm{x}_i}(\bm{x}(k))\bigg)\bigg\|\\
{\leq}&\sum\limits_{j=1}^{n_i}\bigg\|\frac{\partial f_{ij}(\bm{\xi}_{ij}(k))}{\partial \bm{x}_i}{{-}}\frac{\partial f_{ij}(\bm{{x}}(k))}{\partial \bm{x}_i}\bigg\|\\
{\leq}&\sum\limits_{j=1}^{n_i}l_{ij}\|\bm{\xi}_{ij}(k){{-}}\bm{{x}}(k)\|\\
%{\leq}&\sqrt{\sum_{j=1}^{n_i}l_{ij}^2}\sqrt{\sum\limits_{j=1}^{n_i}\|\bm{a}_{j}{{-}}\bm{b}_j\|^2}\\
%\leq&\sqrt{\sum_{j=1}^{n_i}l_{ij}^2}\cdot\big\|\bm{\xi_i}(k){{-}}\bm{1}_{n_i}{{\otimes}}\bm{x}(k)\big\|\\
\leq&\sqrt{\sum_{j=1}^{n_i}l_{ij}^2}\cdot\big\|\bm{e_{\xi_i}}(k)\big\|.
\end{aligned}
\end{equation}
Based on the above formulas, the following can be derived:
\begin{equation*}
\begin{aligned}
&\bar V_x(k{{+}}1){-}\bar V_x(k)\\
\leq&{{-}}\frac{1}{\beta}\sum_{i=1}^N\left(\frac{\partial f_i}{\partial \bm{x}_i}(\bm{x}(k+1)){{-}}\frac{\partial f_i}{\partial \bm{x}_i}(\bm{x}(k))\right)^T\left({\bm{x}}_i(k{{+}}1){{-}}{\bm{x}}_i(k)\right)\\
&{+}\sum_{i=1}^N\frac{1}{n_i}\bigg(\left\|L_i\frac{\partial f_i}{\partial \bm{x}_i}(\bm{x}(k+1)){{-}}L_i\frac{\partial f_i}{\partial \bm{x}_i}(\bm{x}(k))\right\|^2\\
&{+}\left\|L_i\left(\frac{\partial f_i}{\partial \bm{x}_i}(\bm{x}(k)){{-}}n_i\bar{\mathcal{Q}}_i(\bm{\xi}_i({k}))\right)\right\|^2{+}\left\|n_i\breve{L}_i\bm{e_{\psi_i}}(k)\right\|^2\bigg)\\
%=&{{-}}\frac{1}{\beta}(\mathcal{P}(\bm{x}(k{{+}}1)){{-}}\mathcal{P}(\bm{x}(k)))^T({\bm{x}}(k{{+}}1){{-}}{\bm{x}}(k))\\
%&{+}\sum_{i=1}^N\frac{1}{n_i}\bigg(\left\|L_i\frac{\partial f_i}{\partial \bm{x}_i}(\bm{x}(k+1)){{-}}L_i\frac{\partial f_i}{\partial \bm{x}_i}(\bm{x}(k))\right\|^2\\
%&{+}\left\|L_i\left(\frac{\partial f_i}{\partial \bm{x}_i}(\bm{x}(k)){{-}}n_i\bar{\mathcal{Q}}_i(\bm{\xi}_i({k}))\right)\right\|^2{+}\left\|n_i\breve{L}_i\bm{e_{\psi_i}}(k)\right\|^2\bigg)\\
\leq&{{-}}\frac{\mu}{\beta}\left\|{\bm{x}}(k{{+}}1){{{-}}}{\bm{x}}(k)\right\|^2{+}\sum_{i=1}^N\frac{1}{n_i}\bigg(l_i^2\|L_i\|^2\left\|{\bm{x}}(k{{+}}1){{-}}{\bm{x}}(k)\right\|^2\\
&{+}(\sum_{j=1}^{n_i}l_{ij}^2)\left\|L_i\right\|^2 \big\|\bm{e_{\xi_i}}(k)\big\|^2
{+}n_i^2\|\breve{L}_i\|^2\left\|\bm{e_{\psi_i}}(k)\right\|^2\bigg)\\
\leq&{{-}}(\frac{\mu}{\beta}{{-}}\sum_{i=1}^N\frac{l_i^2\|L_i\|^2}{n_i})\beta^2\left\|\hat{L}\hat{\breve{L}}\bm{\psi}({k})\right\|^2\\
&{+}\max_{i\in\mathcal{I}}\Big\{\frac{1}{n_i}(\sum_{j=1}^{n_i}l_{ij}^2)\left\|L_i\right\|^2\Big\}\left\|\bm{e_\xi}(k)\right\|^2\\
&{+}\max_{i\in\mathcal{I}}\{n_i\|\breve{L}_i\|^2\}\left\|\bm{e_\psi}(k)\right\|^2,
\end{aligned}
\end{equation*}
which completes the proof.

%\vspace{1.5ex}
%%%%%%%%%%%%%%%%%%%%%%%%%%%%%%%%%%%%%%%%%%%%%%%%%%%%%%%%%%%%%%%%%%%%%%%%%%%%%%%%%%%%%%%%%%
%%%%%%%%%%%%%%%%%%%%%%%%%%%%%%%%%%%%%%%%%%%%%%%%%%%%%%%%%%%%%%%%%%%%%%%%%%%%%%%%%%%%%%%%%%
%\bibliographystyle{IEEEtran}
%\bibliography{ref_game}

%
%%\begin{thebibliography}{10}
%	\providecommand{\url}[1]{#1}
%	\csname url@samestyle\endcsname
%	\providecommand{\newblock}{\relax}
%	\providecommand{\bibinfo}[2]{#2}
%	\providecommand{\BIBentrySTDinterwordspacing}{\spaceskip=0pt\relax}
%	\providecommand{\BIBentryALTinterwordstretchfactor}{4}
%	\providecommand{\BIBentryALTinterwordspacing}{\spaceskip=\fontdimen2\font plus
%		\BIBentryALTinterwordstretchfactor\fontdimen3\font minus
%		\fontdimen4\font\relax}
%	\providecommand{\BIBforeignlanguage}[2]{{%
%			\expandafter\ifx\csname l@#1\endcsname\relax
%			\typeout{** WARNING: IEEEtran.bst: No hyphenation pattern has been}%
%			\typeout{** loaded for the language `#1'. Using the pattern for}%
%			\typeout{** the default language instead.}%
%			\else
%			\language=\csname l@#1\endcsname
%			\fi
%			#2}}
%	\providecommand{\BIBdecl}{\relax}
%	\BIBdecl

\end{document}